\documentclass{amsart}

\usepackage{latexsym}
\usepackage{enumerate}
\usepackage{amssymb}
\usepackage{amsfonts}
\usepackage{amsmath}
\usepackage{mathrsfs}
\usepackage{amsthm}
\usepackage{geometry}
\usepackage[colorlinks,
linkcolor=red,
anchorcolor=blue,
citecolor=green
]{hyperref}
\usepackage{cleveref}
\usepackage[all]{xy}

\newcommand{\partialbar}{\bar{\partial}}
\newcommand{\zbar}{\bar{z}}
\newcommand{\Ric}{\mathrm{Ric}}
\newcommand{\Vol}{\mathrm{Vol}}
\newcommand{\V}{\mathrm{V}}

\newcommand{\sq}{\backslash}
\newcommand{\diam}{\mathrm{diam}}

\newcommand{\linebundle}{\mathcal{L}}

\newcommand{\dist}{\mathrm{dist}}

\newtheorem{thm}{Theorem}[section]

\newtheorem{prop}[thm]{Proposition}
\newtheorem{coro}[thm]{Corollary}
\newtheorem{lmm}[thm]{Lemma}
\newtheorem{exap}[thm]{Example}

\newtheorem{conj}[thm]{Conjecture}

\theoremstyle{remark} 
\newtheorem*{rmk}{Remark}

\title{	The asymptotic behaviour of Bergman kernels}

\author{Shengxuan Zhou}
\address{Beijing International Center for Mathematical Research\\
Peking University\\
Beijing\\100871\\ China}
\email{zhoushx19@pku.edu.cn}

\begin{document}

\begin{abstract}
Let $\left( X ,d ,p \right) $ be the pointed Gromov-Hausdorff limit of a sequence of pointed complete polarized K\"ahler manifolds $\left( M_l ,\omega_l ,\linebundle_l ,h_l ,p_l \right) $ with $\Ric \left( h_l \right) =2\pi \omega_l $, $\Ric \left( \omega_l \right) \geq -\Lambda \omega_l $ and $\Vol \left( B_1 \left( p_l \right) \right) \geq v $, $\forall l\in\mathbb{N} $, where $\Lambda ,v>0$ are constants. Then $X$ is a normal complex space \cite{lggs1}.

In this paper, we discuss the convergence of the Hermitian line bundles $( \linebundle_l ,h_l )$ and the Bergman kernels. In particular, we show that the K\"ahler forms $\omega_l $ converge to a unique closed positive current $\omega_X $ on $X_{reg}$. By establishing a version of $L^2$ estimate on the limit line bundle on $X$, we give a convergence result of Fubini-Study currents on $X$. Then we prove that the convergence of Bergman kernels implies a uniform $L^p$ asymptotic expansion of Bergman kernel on the collection of $n$-dimensional polarized K\"ahler manifolds $(M,\omega ,\linebundle ,h)$ with Ricci lower bound $-\Lambda$ and non-collapsing condition $\Vol\left( B_1 (x) \right) \geq v >0 $. Under the additional orthogonal bisectional curvature lower bound, we will also give a uniform $C^0$ asymptotic estimate of Bergman kernel for all sufficiently large $m$, which improves a theorem of Jiang \cite{jws1}. By calculating the Bergman kernels on orbifolds, we disprove a conjecture of Donaldson-Sun in \cite{donsun1}.
\end{abstract}

\maketitle

\tableofcontents

\section{Introduction}
\label{intro}

Let $X$ be a reduced normal complex space, $X_{reg} $ be the set of regular points of $X$, $\mu $ be a Radon measure defined on $X_{reg} $, $X_{sing} =X\sq X_{reg} $, and $\linebundle $ be a holomorphic line bundle on $X_{reg} $ with a continuous Hermitian metric $h$. Then we can define the Bergman space for each $m\in\mathbb{N}$ as the following linear space
$$ H^0_{L^2} \left( X ,\mathcal{L}^m \right) = \left\lbrace s\in H^0 \left( X_{reg} ,\linebundle^m \right) : \int_{X_{reg}} \left\Vert s \right\Vert_{h^m}^2 d\mu < \infty \right\rbrace ,$$
with an $L^2 $ inner product $\left\langle s_1 , s_2 \right\rangle_{L^2 ;X,\linebundle ,h,m} = \int_{X_{reg}} \left\langle s_1 , s_2 \right\rangle_{h^m}^2 d\mu $ and an $L^2$ norm $\left\Vert s \right\Vert_{L^2 ;X,\linebundle ,h,m} = \left( \left\langle s,s \right\rangle_{L^2 ;X,\linebundle,h,m} \right)^{\frac{1}{2}} $. Sometimes we denote them briefly by $\left\langle s_1 , s_2 \right\rangle_{L^2} $ and $\left\Vert s \right\Vert_{L^2} $, respectively. Hence the Bergman kernel can be defined as
$$ \rho_{X, \mu ,\linebundle,h,m} (x) = \sup_{\left\Vert s \right\Vert_{L^2} =1 }  \left\Vert s(x) \right\Vert_{h^m}^2 ,\;\; \forall x\in X_{reg}. $$
The Bergman kernels will also be abbreviated as $\rho_{m}$ when we do not emphasize the structure on $X$.

An important case of Bergman kernel is when $(M,\omega ,\linebundle ,h)$ is a polarized manifold. In this case, $M$ is an $n$-dimensional complex manifold, $ \omega $ is a complete K\"ahler metric on $M$, $\linebundle $ is a holomorphic line bundle on $M$ equipped with a Hermitian metric $h$ whose curvature form is $2\pi\omega $, and the measure $\mu $ on $M$ is obtained by the volume form $ \frac{\omega^n}{n!} $. We say that the Hermitian holomorphic line bundle $(\linebundle ,h)$ is a polarization of $(M,\omega ) $ in this case.

The Bergman kernel plays an important role in K\"ahler geometry. In the pioneering work \cite{tg1}, Tian used his peak section method to prove that Bergman metrics converge to the original polarized metric in the $C^2$-topology, and the K\"ahler potential of the Bergman metric is the Bergman kernel. By a similar method, Ruan \cite{wdr1} proved that this convergence is $C^\infty $. Later, Zelditch \cite{sz1}, also Catlin \cite{cat1} independently, used the Szegö kernel to obtain an alternative proof of the $C^\infty $-convergence of Bergman metrics and they gave the following asymptotic expansion of Bergman kernel on a fixed polarized K\"ahler manifold $M$:
\begin{eqnarray}
\rho_{\omega  ,m} (x) \sim m^n +a_1 (x) m^{n-1} +a_2 (x) m^{n-2} +\cdots ,\label{tyzasymptoticexpansion}
\end{eqnarray}
where $ \rho_{\omega ,m} $ is the Bergman kernel on $(M,\omega ,\linebundle ,h )$, and $a_j $ are smooth coefficients. This expansion can be also obtained by Tian's peak section method (see \cite{liuzql1}) and is often called Tian-Yau-Zelditch expansion. By using the heat kernel, Dai-Liu-Ma \cite{dailiuma1} gave another proof of the Tian–Yau–Zelditch expansion, and moreover, they also considered the asymptotic behavior of Bergman kernels on symplectic manifolds and Kähler orbifolds (see also Ma-Marinescu’s book \cite{mamari1}). Lu \cite{zql1} used the peak section method to calculate the lower order terms in the asymptotic expansion of Bergman kernels. In particular, $a_1 = \frac{S(\omega )}{2} $, where $S(\omega )$ is the scalar curvature of $(M,\omega )$. There are many important applications of the Tian–Yau–Zelditch expansion, for example, \cite{don1} and \cite{kwz2}.

Let $\left( X ,d ,p \right) $ be the Gromov-Hausdorff limit of a sequence of pointed complete polarized K\"ahler manifolds $\left( M_l ,\omega_l ,p_l \right) $ with $\Ric \left( \omega_l \right) \geq -\Lambda \omega_l $ and $\Vol \left( B_1 \left( p_l \right) \right) \geq v $, $\forall l\in\mathbb{N} $, where $\Lambda ,v>0$ are constants. Recall that Liu-Sz\'ekelyhidi \cite{lggs1} has proved that there exists a subsequence of $\left( M_l ,\omega_l ,p_l \right) $ such that the complex structures converge to a normal complex space structure on $X$. Without loss of generality, we can assume that the complex structures of $M_l $ converge to a normal complex space structure on $X$. Then we can define the Bergman kernel on $X$. 

Our first result describes the convergence of the polarizations and Bergman kernels. It can be seen as a metric version of the flatness for families
of projective varieties when $X$ is compact. Note that the following result is proved in \cite{tg5, tg7, tg6} when the sequence of K\"ahler manifolds are K\"ahler-Einstein manifolds, and in \cite{tg3} when they are conic K\"ahler-Einstein manifolds.

\begin{thm}
\label{thmcontinuousbergmankernel}
Let $\left( \linebundle_l ,h_l \right) $ be a sequence of Hermitian line bundles on $M_l $ such that $\Ric \left( h_l \right) = 2\pi \omega_l $. Then we can find a sequence of integers $l_j \to\infty $, such that $\left( \linebundle_{l_j} , h_{l_j} \right) $ converges to a holomorphic line bundle $\left( \linebundle_{\infty } , h_{\infty } \right) $ on $X_{reg } $, and for each $x\in X$, there exists an integer $D\in\mathbb{N}$, such that $\left( \linebundle_\infty^D ,h_\infty^D \right) $ can be extended to a Hermitian holomorphic line bundle on a neighborhood of $x\in X$. 

Moreover, for any $m> \frac{\Lambda}{2\pi} $, the Bergman kernels $\rho_{ M_{l_j} ,\mu , \linebundle_{l_j} , h_{l_j} , m} $ converge to the Bergman kernel $\rho_{X,\mu ,\linebundle_\infty ,h_\infty ,m} $, where $\mu $ is the $2n$-dimensional Hausdorff measure of the corresponding space.
\end{thm}

The proof of the convergence of Bergman kernels in Theorem \ref{thmcontinuousbergmankernel} is to construct a sequence of holomorphic sections of $\linebundle_{l_j} $ to approximate a holomorphic section of $\linebundle_\infty $. In Theorem \ref{thmcontinuousbergmankernel}, the condition $2\pi m> \Lambda $ seems like a technique condition, but it is necessary even when $\left( X,\linebundle_\infty \right)$ is a compact polarized K\"ahler manifold. Example \ref{genus2example} is a counterexample when $2\pi m=\Lambda $. 

The next theorem shows that the limit Hermitian metric $h_\infty $ on the limit line bundle $\linebundle_\infty $ can be uniquely determined up to multiplication by a pluriharmonic function. See Corollary \ref{corouniquehermitianmetric}.

\begin{thm}
\label{thmuniquenesskahlerpotential}
Let $\left( X ,d ,p \right) $ be a pointed Gromov-Hausdorff limit as in Theorem \ref{thmcontinuousbergmankernel}. Then the K\"ahler potentials on $M_l $ converge to a unique K\"ahler space structure on $X $, up to equivalence, and the K\"ahler forms $\omega_l $ converge to a unique closed positive current $\omega_X $ on $X$.
\end{thm}

\begin{rmk}
In general, it is not clear whether the Gromov-Hausdorff limit $X$ is a complex space without polarization. So through the argument in Theorem \ref{thmuniquenesskahlerpotential}, we can only get the K\"ahler space structure on $X\sq \Sigma_\epsilon $ for some $\epsilon $, where $\Sigma_\epsilon $ is the set of points $y\in X$ satisfying that $\lim_{t\to 0^+ } t^{-2n} \Vol \left( B_t (y) \right) \leq \V_{2n } -\epsilon $, $\forall \epsilon >0$, and $\V_{2n } $ is the volume of the unit ball in $\mathbb{C}^n $.
\end{rmk}

Theorem \ref{thmuniquenesskahlerpotential} shows that there exists a unique positive current $\omega_X $ on $X$ that is compatible with the metric structure and complex structure of $X$ in the above way. If $\omega_X$ is smooth, then it is a K\"ahler form and $X$ becomes a K\"ahler manifold. But $\omega_X $ is only a closed positive current in general. By the convergence of Hausdorff measures \cite[Theorem 5.9]{chco3} and currents \cite[Corollary II-3.6]{dm1}, one can see that $\frac{\omega^n_X}{n!}$ can be represented by the $2n$-dimensional Hausdorff measure of $(X,d)$. By regularizing the Hermitian metric, we can establish a version of $L^2$-estimate for limit line bundles. See Proposition \ref{propl2estimateghlimit}. Then we can get the convergence of Fubini-Study currents on the limit space. Let us state our setting now.

Let $(X,\omega_X ,p )$ be the K\"ahler normal space as in Theorem \ref{thmuniquenesskahlerpotential}, and $\linebundle'_\infty $ be a line bundle on $X$ with a continuous Hermitian metric $h'_\infty $. Then we can define the curvature currents as $c_1 (\linebundle'_\infty ) = \frac{1}{2\pi} \Ric (h'_\infty ) $. Next, we define the Fubini-Study currents on $X$ by 
$$ \gamma_m = \frac{\sqrt{-1}}{2\pi} \partial\partialbar \log \left( \rho_{\linebundle'_\infty ,m } \right) +mc_1 (\linebundle'_\infty ) .$$
Recall that $\gamma_m \to c_1 (\linebundle'_\infty )  $ in the $C^\infty $-topology when $(X , \omega_X )$ is a smooth K\"ahler manifold and $c_1 (\linebundle'_\infty ) =\omega_X $ \cite{tg1,wdr1}. In this case, $ \gamma_m $ is called the Bergman metric on $X$. In the weak sense, Coman-Ma-Marinescu \cite[Theorem 1.1]{comamari1} and Coman-Marinescu \cite[Theorem 5.1]{comari1} considered the convergence of Fubini-Study currents. But they require the measure of the base space to be given by a smooth K\"ahler form.

Now we can state our result about the convergence of Fubini-Study currents. This is an analogy of the convergence result of Fubini-Study currents \cite{comamari1, comari1}. Note that the measure $\mu $ here is not generated by smooth K\"ahler form in general.

\begin{prop}
\label{propfscurrentdistribution}
Suppose that $\linebundle'_\infty $ is the limit line bundle of a sequence of line bundles on $\left( M_l ,\omega_l \right) $, and $\Ric (h'_\infty ) \geq \epsilon \omega_X $ for some constant $\epsilon >0$. Then we have $ m^{-1} \log \left( \rho_{\linebundle'_\infty ,m } \right) \to 0 $ in $L^\infty_{loc} (X) $.

Moreover, we have $\frac{1}{m^k} \gamma^k_m \to c_1 (\linebundle'_\infty )^k $ in the weak sense of currents on $X$, $k=1,\cdots ,n$.
\end{prop}

\begin{rmk}
Under the additional conditions of orthogonal bisectional curvature lower bound and non-collapsing everywhere, the line bundle $\linebundle'_\infty $ need not be a limit line bundle. See Corollary \ref{corol2thmlimit}.
\end{rmk}

As an application of Theorem \ref{thmcontinuousbergmankernel}, we can obtain the following $L^p$ asymptotic expansion of the Bergman kernel. This is a first-order uniform Tian–Yau–Zelditch expansion of the $L^p$ version. Note that the uniform $C^0$ asymptotic expansion does not hold in this case. 
\begin{prop}
\label{lpestprop}
Let $p\geq 1$, $\Lambda ,v >0$ and $n\in\mathbb{N}$ be given. Suppose $(M,\omega ,\linebundle ,h)$ is a complete polarized K\"ahler manifold and $x\in M$. If $\Ric(\omega ) \geq -\Lambda \omega $ and $\Vol\left( B_1 (x) \right) \geq v $, then we have
$$ \left\Vert m^{-n} \rho_{m} - 1 \right\Vert_{L^p \left( B_1 (x) \right) } \leq \Psi \left( m^{-1} \big| p,\Lambda ,v \right) ,$$
where $\left\Vert f \right\Vert_{L^p \left( \Omega \right) } = \left( \frac{1}{\Vol(\Omega )}\int_{\Omega} |f|^p \right)^\frac{1}{p} $.
\end{prop}

In this paper, $\Psi \left( \epsilon_1 ,\cdots ,\epsilon_m \big| u_1 ,\cdots ,u_n \right)$ denotes a function of $\epsilon_1 ,\cdots ,\epsilon_m ,u_1 ,\cdots ,u_n $, such that for fixed $u_i$ we have $ \lim_{\epsilon_1 ,\cdots ,\epsilon_m \to 0} \Psi =0 $. 

For a fixed $(M,g)$, we see that the order of the asymptotic expansion (\ref{tyzasymptoticexpansion}) can be arbitrarily high. But in the uniform sense, even if we only consider the $L^1 $ norm, it is impossible to give a uniform second-order asymptotic expansion under the conditions above. See Example \ref{oscillation}.

Our argument is to divide the manifold $(M,\omega )$ into two parts, and then estimate the Bergman kernels on it respectively. On one part of $(M,\omega )$, the Bergman kernel $\rho_{\omega ,m}$ is sufficiently close to $m^n$, and the remaining part of $(M,\omega )$ is small enough.

An interesting and famous question about Bergman kernels is to ask whether there is a uniform positive lower bound for Bergman kernels on a certain class of polarized K\"ahler manifolds. Such results are often called Tian’s partial $C^0 $-estimates. Tian’s partial $C^0 $-estimate originated from Tian's work on finding K\"ahler-Einstein metrics on del Pezzo surfaces \cite{tg5}. It has been a powerful tool in solving problems in K\"ahler geometry, especially in the proof of the Yau-Tian-Donaldson conjecture \cite{tg6,tg3,chendonsun1}. There are many works on Tian’s partial $C^0$-estimate \cite{donsun1,jws1,lggs1,sze1,wfzxh1,kwz1}.

Our next result is a refinement of Tian’s partial $C^0$-estimate when the orthonormal bisectional curvature has a lower bound. Here we recall the meaning of orthonormal bisectional curvature. Given a constant $K\in\mathbb{R}$ and a K\"ahler manifold $(M,\omega )$, we say that the orthogonal bisectional curvature is greater than or equal to $K$ $( OB\geq K )$, if and only if 
$$ R\left( X,\bar{X} , Y, \bar{Y} \right) \geq K \left( \left\Vert X \right\Vert^2 \left\Vert Y \right\Vert^2 +\left| \left\langle X,\bar{Y} \right\rangle \right|^2 \right) , $$
for any $X,Y\in T^{1,0} M $ satisfying $\omega (X,\bar{Y} ) =0 $, where $R$ is the curvature operator.

\begin{prop}
\label{refinementofbkk}
Given $K\in \mathbb{R}$, $v >0$, $n\in\mathbb{N}$, there are constants $b, m_0 >0$ with the following property. Let $(M,\omega ,\linebundle ,h)$ be a polarized K\"ahler manifold with $\Ric (\omega ) \geq -\Lambda \omega $, $OB\geq -\lambda $ and $\Vol\left( B_1 (x) \right) \geq v $ for some $x\in M$. Then
$$  m^{-n} \rho_{\omega ,m} (x) \geq b , \;\; \forall m\geq m_0 .$$
\end{prop}

By using the theory of K\"ahler-Ricci flow, Jiang \cite{jws1} proved the result above when $M$ is Fano, $\linebundle =K_{M}^{-1} $, and the bisectional curvature is nonnegative. Note that $\Ric (\omega ) \geq 0 $ and $OB\geq 0$ when the bisectional curvature is nonnegative. Our method is to directly apply H\"ormander's $L^2$ method to construct a suitable section, so we don't need to assume $\linebundle =K_{M}^{-1}$.

In general, the partial $C^0$-estimate does not hold for all sufficiently large $m$ \cite{tg5, dailiuma1}, but holds for an increasing sequence $m_j \to\infty $ \cite{lggs1, kwz1}. It is speculated that Tian's partial $C^0$-estimate characterized the geometry of the polarized manifold $(M,\linebundle )$ in an effective way. For example, like Matsusaka's big theorem's refinement of Kodaira embedding theorem, the partial $C^0$-estimate characterizes the finitely generativeness of the ring $ R(X) =\oplus_{k\geq 0} H^0 \left( X,\linebundle^k \right)  $ \cite{chili1}.

Under the two-sided Ricci curvature bounds and diameter bound, Donaldson-Sun conjectures that Tian's partial $C^0$-estimate is related to the volume ratio in \cite{donsun1}. They also conjectured that the lower bound could be infinitely close to $1$. 

\begin{conj}[\cite{donsun1}, Conjecture 5.15]
\label{dsconj}
Given $n\in\mathbb{N}$, $c,d>0$ and $\eta \in (0,1) $, there is a number $m_0$ with the following property. Let $(M,\omega ,\linebundle ,h)$ be a polarized K\"ahler manifold with $\left| \Ric \right| \leq 1 $, $\diam (M,\omega) \leq d $. Suppose that $\Vol \left( B_r \right) \geq c\frac{\pi^n}{n!} r^{2n} $ holds for each metric $r$-ball in $M$, $\forall r\in (0,d) $. Then we have
$$ \eta \leq (mD)^{-n} \rho_{mD} \leq \frac{1}{c\eta} ,\;\; \forall m>m_0 ,$$
where $D=D(c)$ is the least integer such that all integers less than or equal to $c^{-1}$ divide $D$.
\end{conj}

Unfortunately, the lower bound cannot always be sufficiently close to $1$ by increasing $m$. We give a counterexample as follows. The idea is to calculate the Bergman kernels on some complex spaces with non-trivial singularity and then use K\"ahler manifolds to approximate it. Actually, we show the following results by approximating a special orbifold.

\begin{thm}
\label{dscounterexp}
Given an integer $n\geq 2 $, there are constants $m_0 , \epsilon >0$ with the following property.

For each $m>m_0 $, we can find an $n$-dimensional polarized K\"ahler manifolds $\left( M ,\omega ,K_M^{-1} ,h \right) $ such that $\Ric \left( \omega \right) = 2\pi \omega $, and the $m$-th Bergman kernel satisfying that
$$ \inf_{x\in M} \rho_{\omega ,m } \left( x \right)  \leq ( 1 -\epsilon )m^n .$$
\end{thm}

\begin{rmk}
By the Bonnet-Myers Theorem, $\Ric \left( \omega \right) = 2\pi \omega $ implies that $\diam (M,\omega ) < 2\pi $. 
\end{rmk}

By choosing $\eta >1-\epsilon $, we disprove the lower bound in Donaldson-Sun's conjecture. 

\begin{coro}
The lower bound in Conjecture \ref{dsconj} is false.
\end{coro}

This paper is organized as follows. In Section \ref{Preparatoryresults}, we collect some preliminary results that will be used many times. Then we will prove Theorem \ref{thmuniquenesskahlerpotential} in Section \ref{limitlinebundlesection}. The proof of Theorem \ref{thmcontinuousbergmankernel} is presented in Section \ref{convergencebergmankernelsection}. In Section \ref{l2estghlimitsection}, we give a version of $L^2$ estimate on the Gromov-Hausdorff limit. Finally, the proofs of Proposition \ref{lpestprop}, Proposition \ref{refinementofbkk} and Theorem \ref{dscounterexp} are contained in Section \ref{applicationsection}.

\vspace{0.2cm}

\textbf{Acknowledgement.} The author wants to express his deep gratitude to Professor Gang Tian for suggesting this problem and constant encouragement. He would also be grateful to Bojie He, Wenshuai Jiang, Gan Li, Minghao Miao, Feng Wang, Kewei Zhang and Ziyi Zhao for many valuable comments. The author also thanks Zexing Li for reading the earlier version carefully, pointing out mistakes and typos, and giving many suggestions about writing the article.

\section{Preliminaries}
\label{Preparatoryresults}

We recall here some basic notations and results that will be needed throughout the paper.

\subsection{The metrics of Gromov-Hausdorff limits}

    We need the Gromov-Hausdorff convergence theory in this paper. For the deﬁnition and basic properties of Gromov-Hausdorff convergence, we refer to \cite{dbybsi1}, \cite{pp1}. This theory has been developed by Anderson, Cheeger, Colding, Gromov, Jiang, Naber, Tian and others. Here we only mention some recent results. Given a metric space $X$ we say that a metric ball $B_r (x) \subset X$ is $(k,\epsilon )$-$symmetric$ if $d_{GH} \left( B_r(x) ,B^{\mathbb{R}^k \times C(Z) }_r ( (0^k ,o) ) \right) < \epsilon r $ for some metric cones $\mathbb{R}^k \times C(Z) $. We define the $r$-$scale$ $(k,\epsilon )-singular$ $set$ $\mathcal{S}^k_{\epsilon ,r}$ to be
$$\mathcal{S}^k_{\epsilon ,r} (X)= \left\lbrace \; x\in X : \textrm{ for no $s\in (r,1] $ is } B_s (x) \textrm{ a $(k+1 ,\epsilon )$-symmetric ball } \right\rbrace .  $$
The following Minkowski type estimate is proved by Cheeger-Jiang-Naber \cite[Theorem 1.7]{jcwsjan1}.
\begin{thm}[\cite{jcwsjan1}, Theorem 1.7]
\label{jcwsjan1thm17coro}
For each $\epsilon >0$, there exists $C=C(n,v,\epsilon )>0$ such that the following estimate holds. Let $\left( M_i ,g_i ,p_i \right)$ be a sequence of $n$-dimensional pointed Riemannian manifolds converging to $(X,d,p)$ in the Gromov-Hausdorff sense with $\Ric\left( g_i \right) \geq -(n-1)\; g_i $ and $ \Vol\left( B_{1} \left( p_i \right) \right) \geq v $. Then
$$ \Vol \left( B_r \left( \mathcal{S}^k_{\epsilon ,r} (X) \right) \cap B_1 (x)  \right) \leq Cr^{n-k} .$$
\end{thm}

\subsection{Convergence of line bundles and K\"ahler spaces}

Let $\left( M_l ,\omega_l , p_l \right) $ be a sequence of pointed complete K\"ahler manifolds converges to $(X,d,p)$. Suppose that there is a complex space structure on the topological space $X$ (see \cite{dm1}, Definition 5.2 ). We say that the sequence of complex structures on $M_l $ converges to $X $, if for each $x\in X $, we can find a constant $r>0 $, a sequence of points $x_l \in M_l $ and holomorphic maps $ F_l = \left( z_{l,1} , \cdots ,z_{l,N} \right) : B_{r} (x_l ) \to \mathbb{C}^N $, such that $x_l $ converges to $x$, and $F_l $ converges to an injective holomorphic map $F$ on $B_{r} (x )$. By the proper mapping theorem, these $F_l $ and $F$ can be chosen as holomorphic charts for $x\in X_{reg} $. 

We further assume that $ \left( \linebundle_l ,h_l \right) $ be a sequence of Hermitian holomorphic line bundles on $M_l $, and $\linebundle_\infty $ be a holomorphic line bundle on $X $ with a continuous Hermitian metric $h_\infty $. We say that the sequence of holomorphic line bundles $ \linebundle_l $ converges to $\linebundle_{\infty}  $, if there are a sequence of open coverings $\left\lbrace U_{j,l} \right\rbrace_{j\in J} $ of $M_l $ and an open covering $\{ U_j \}_{j\in J} $ of $X$, satisfying that $\linebundle_l \big|_{U_{j,l}} $ is trivial, $\forall j,l $, and we can find local frames $s_{j,l} \in H^0 \left( U_{j,l} , \linebundle_l \right) $ such that the transition functions $f_{j,k,l} = s_{j,l} s_{k,l}^{-1} \in \mathscr{O} \left( U_{j,l} \cap U_{k,l} \right) $ converges to the transition functions $f_{j,k,\infty } $ of $\linebundle_\infty $, as $l\to\infty $. When $\left\Vert s_{j,l} \right\Vert_{h_l}  $ converges to $\left\Vert s_{j,\infty } \right\Vert_{h_\infty }  $, we say that the Hermitian line bundles $ \left( \linebundle_l ,h_l \right) $ converge to $ ( \linebundle_\infty ,h_\infty ) $. In the same way, we can define the convergence of line bundles on an open subset of $X$.

The following theorem is proved in Liu-Sz\'ekelyhidi \cite{lggs1}.

\begin{thm}[\cite{lggs1}, Theorem 1.1]
\label{complexstructurethmricci}
Let $\left( M_l ,\omega_l ,\linebundle_l ,h_l ,p_l \right) $ be a sequence of pointed complete polarized K\"ahler manifolds satisfying that $\Ric \left( \omega_l \right) \geq -\Lambda \omega_l $ and $\Vol \left( B_1 \left( p_l \right) \right) \geq v $, $\forall l\in\mathbb{N} $, where $\Lambda ,v>0$ are constants. Assume that $ \left( M_l ,\omega_l ,p_l \right)$ converges to $\left( X,d ,p_\infty \right) $ in the pointed Gromov-Hausdorff topology. Then $X$ has a structure of normal analytic variety such that there exists a subsequence of complex structures on $M_l$ converges to $X$.
\end{thm}

The key step of Liu-Sz\'ekelyhidi's proof of Theorem \ref{complexstructurethmricci} is to establish the following local partial $C^0$-estimate.

\begin{prop}[\cite{lggs1}, Proposition 3.1]
\label{localpartialc0estimate}
Given $v,\xi >0$ there are $K,\epsilon ,C>0 $ with the following property. Let $\left( M ,\omega ,\linebundle ,h \right) $ be a complete polarized K\"ahler manifold satisfying that $\Ric \left( \omega \right) \geq -\epsilon \omega $ and $\Vol \left( B_1 \left( p \right) \right) \geq v $. Suppose that $d_{GH} \left( B_{\epsilon^{-1} } (p) , B_{\epsilon^{-1} } (o) \right) \leq\epsilon $ for a metric cone $(V,o)$. Then $\linebundle^m $ admits a global section $s\in H_{L^2}^0 \left( M,\linebundle^m \right) $ such that $\left\Vert s \right\Vert_{L^2} \leq C $, and $ \left| \left\Vert s \right\Vert^2 - e^{-m\pi d(x,p) } \right| \leq \xi $ for $x\in B_{1} (p) $.
\end{prop}

As in \cite{egz1} and \cite{lott1}, a K\"ahler space consists of a complex space with an open covering $\left\lbrace U_j \right\rbrace_{j=1}^{\infty} $ and a sequence of continuous plurisubharmonic functions $\psi_j $ on $U_j $, such that $\psi_j - \psi_{j'} $ is pluriharmonic on $U_j \cap U_{j'} $ if $U_j \cap U_{j'} \neq \emptyset $. Two such collections $\lbrace ( U_j ,\psi_j ) \rbrace $ and $\lbrace ( \widehat{U}_k ,\widehat{\psi}_k ) \rbrace $ are equivalent if $\psi_j -\widehat{\psi}_k $ is pluriharmonic on $U_j \cap \widehat{U}_k $ when $U_j \cap \widehat{U}_k \neq \emptyset $. 

Let $\left( M_l ,\omega_l , p_l \right) $ be a sequence of pointed complete K\"ahler manifolds converges to $(X,d,p)$. Suppose that there is a K\"ahler space structure $\lbrace ( U_j ,\psi_j ) \rbrace $ on $V\subset X$. We say that the K\"ahler potentials converge to the K\"ahler space structure $\lbrace ( U_j ,\psi_j ) \rbrace $ in the Gromov-Hausdorff sense, if for each $x\in U_j \subset V$, there exists a constant $r>0$, a sequence of points $x_l \to x $ and a sequence of K\"ahler potentials $\psi_{l,x} $ on $B_r \left( x_l \right) $, such that $\psi_{l,x} $ converge to a pluriharmonic function $\psi_{\infty ,x}$ on $B_r ( x) \subset V $, and $\psi_{\infty ,x} -\psi_j $ is pluriharmonic.

\subsection{H\"omander's \texorpdfstring{$L^2$}{Lg} estimate}
    Now we state the H\"omander's $L^2$ estimate as follows without proof. The proof can be found in \cite{dm1} and \cite{lh1}. We note that the $L^2$ estimate also holds for complete K\"ahler orbifold and weakly pseudoconvex manifold (the metric is not assumed to be complete).
\begin{prop}
\label{l2m}
Let $(M,\omega )$ be an $n$-dimensional complete K\"ahler manifold. Let $(\linebundle ,h)$ be a Hermitian holomorphic line bundle, and $\psi $ be a function on $M$, which can be approximated by a decreasing sequence of smooth function $\left\lbrace \psi_i \right\rbrace_{i=1}^{\infty} $. Suppose that $$\sqrt{-1}\partial\bar{\partial} \psi_i + \Ric(\omega )+\Ric(h)  \geq c\omega $$
for some positive constant $c >0 $. Then for any $\linebundle $-valued $(0,q)$-form $\zeta\in L^{2}$ on $M$ with $\bar{\partial} \zeta =0$ and $\int_{M} ||\zeta ||^{2} e^{-\psi} \omega^n $ finite, there exists an $\linebundle $-valued $(0,q-1)$-form $u\in L^{2}$ such that $\partialbar u=\zeta$ and $$\int_{M} \Vert u\Vert^{2} e^{-\psi } \omega^n \leq \int_{M} c^{-1} \Vert \zeta \Vert^{2} e^{-\psi} \omega^n ,$$
where $||\cdot ||$ denotes the norms associated with $h$ and $\omega $, and $q=1,\cdots ,n$.
\end{prop}

\section{Limits of polarizations and K\"ahler potentials}
\label{limitlinebundlesection}

In this section, we consider the convergence of polarizations and K\"ahler potentials. Then we will prove the uniqueness of the limit of K\"ahler potentials.

Let $\left( M_l ,\omega_l ,\linebundle_l ,h_l ,p_l \right) $ be a sequence of pointed complete polarized K\"ahler manifolds satisfying that $\Ric \left( h_l \right) = 2\pi \omega_l $, $\Ric \left( \omega_l \right) \geq -\Lambda \omega_l $ and $\Vol \left( B_1 \left( p_l \right) \right) \geq v $, $\forall l\in\mathbb{N} $, where $\Lambda ,v>0$ are constants. Suppose that $ \left( M_l ,\omega_l ,p_l \right) \to \left( X ,d ,p \right) $ in the pointed Gromov-Hausdorff topology, and the normal complex space structure of $X$ is the limit of the sequence of complex structures of $M_l$. Then we can construct the limit of the Hermitian line bundles. The following proposition is a corollary of the local partial $C^0 $-estimate.

\begin{prop}
\label{constructionlinebundleprop}
Under the above assumptions, there exists a sequence of integers $l_j \to\infty $, such that $\left( \linebundle_{l_j} , h_{l_j} \right) $ converges to a holomorphic line bundle $ \linebundle_{\infty } $ with a continuous Hermitian metric $h_\infty $ on $X_{reg } $. Moreover, for each $x\in X$, there exists an integer $D\in\mathbb{N}$, such that $\left( \linebundle_\infty^D ,h_\infty^D \right) $ converge to a Hermitian holomorphic line bundle on a neighborhood of $x\in X$.
\end{prop}

\begin{proof}
Since the sequence of complex structures is convergence, for each $y_\infty \in X_{reg} $, there exist a constant $r_{y,1} >0$ and a sequence of coordinates $z_l = \left( z_{l,1} ,\cdots ,z_{l,n} \right) : B_{r_{y,1} }  \left( y_{l} \right) \to \mathbb{C}^n $ around $ y_l $, such that $y_l \to y_\infty $ in the Gromov-Hausdorff sense, and the sequence of coordinates $z_l $ converge to a holomorphic coordinate $z_\infty = \left( z_{\infty ,1} ,\cdots ,z_{\infty ,n} \right) $ on $ B_{r_{y,1}}  \left( y_\infty \right) \subset X_{reg} $.

By the relative volume comparison and the Cheeger-Colding's ``almost volume cone implies almost metric cone" theorem, we can find a decreasing sequence of positive constants $\left\lbrace r_{y,k} \right\rbrace_{k=2}^{\infty } $ such that $r_{y,k} \leq \min \left\lbrace k^{-2} , r_{y,1} \right\rbrace  $, and there exists a sequence of metric cones $\left( V_k ,o_k \right) $ satisfying that $ d_{GH} \left( B_{r_{y,k}} \left( y_\infty \right) , B_{r_{y,k}} \left( o_k \right) \right) \leq k^{-1} r_{y,k} $. Now we can find constants $D_y ,l_y ,k_y >0$ and holomorphic sections $s_{k,y_l } \in H^0 \left( B_{r_{y,k_y} } \left( y_l \right) ,\linebundle^{D_y} \right)  ,$ $\forall l\geq l_y $ by the Liu-Sz\'ekelyhidi's local partial $C^0$-estimate, such that $2^{-1} \leq \left\Vert s_{k,y_l } \right\Vert \leq 2 $ on $B_{r_{y,k_y} } \left( y_l \right) $. Recall that $ B_{r_{y,k_y} } \left( y_\infty \right) $ is biholomorphic to an open domain in $\mathbb{C}^n $, then there exists a constant $r_y >0$ such that $B_{3 r_{y } } \left( y_\infty \right) $ is contractible in $ B_{r_{y,k_y} } \left( y_\infty \right) $. Hence we can apply the method of analytic continuation to construct a sequence of holomorphic sections $ s_{y_l } \in H^0 \left( B_{2 r_{y } } \left( y_l \right) ,\linebundle \right) $ such that $s^{D_y}_{y_l} = s_{k,y_l } $.

For each compact subset $K\subset X_{reg} $, we can find a finite subset of $\left\lbrace B_{r_y } (y) \right\rbrace_{y\in K } $ covering $K$, and the constants $r_y $ are the constants in the above argument. If $y'_{\infty} \in K $ such that $B_{r_{y'} } \left( y'_{\infty} \right) \cap B_{r_y} \left( y_\infty \right) \neq \emptyset  $, then we can use the above argument to find a sequence of points $y'_l \to y'_{\infty } $ and a sequence of holomorphic sections $s_{y'_l} \in H^0 \left( B_{2 r_{y' } } \left( y'_l \right) ,\linebundle \right) $ such that $2^{-1} \leq \Vert s_{ y'_l } \Vert \leq  2 $ on $B_{2r_{y'_l} } \left( y'_l \right) $. Then the modulus of the transition function $ | f_{y_l ,y'_l} | = \Vert s_{y_l} s^{-1}_{y'_l} \Vert \leq 4 $ on $B_{2r_{y'} } \left( y'_{l} \right) \cap B_{2r_y} \left( y_l \right) $. Taking a subsequence, we can assume that $f_{y_l ,y'_l} $ and $f_{y'_l ,y_l} $ converge to holomorphic functions on $B_{r_{y'} } \left( y'_{\infty} \right) \cap B_{r_y} \left( y_\infty \right)$, and the sequence of line bundles is convergence on $B_{r_{y'} } \left( y'_{\infty} \right) \cap B_{r_y} \left( y_\infty \right) $. Hence we can construct the limit $\left( \linebundle_\infty ,h_\infty \right) $ of a subsequence of the sequence of line bundles $\left( \linebundle_l ,h_l \right) $ on $K$ in a similar way. By choosing an exhaustive sequence of compact subsets of $X_{reg}$, we can find a sequence of Hermitian holomorphic line bundles $\left( \linebundle_{l_j} ,h_{l_j} \right) $ converges to a holomorphic line bundle $\left( \linebundle_\infty ,h_\infty \right) $ on $X_{reg}$. 

Now we consider the convergence of $\linebundle^D_l $ around the point $x\notin X_{reg}$. Without loss of generality, we can assume that $\left( \linebundle_l ,h_l \right) \to \left( \linebundle_\infty ,h_\infty \right) $ on $X_{reg}$. By a similar argument, we can find a sequence of points $x_l \to x $, two constants $r,D >0$, and a sequence of holomorphic sections $s_{ l } \in H^0 \left( B_{2r } \left( x_l \right) ,\linebundle_l^{D} \right)  ,$ such that $\left\Vert  s_l \right\Vert \in \left( 2^{-1} ,2 \right) $ on $ B_{2r } \left( x_l \right) $. Then we can construct the limit line bundle of the sequence $\left( \linebundle_{l_j}^D , h_{l_j}^D \right) $ on $B_r (x) \cup X_{reg} $ by taking subsequence if necessary.
\end{proof}

\begin{rmk}
In the proof of Proposition \ref{constructionlinebundleprop}, the completeness of $M_l $ is only used when applying the partial $C^0 $-estimate. Actually, we can make such constructions on Stein manifolds.
\end{rmk}

The following lemma gives the construction of the K\"ahler space structure on the limit space.

\begin{lmm}
\label{lmmexistkahlerspace}
Let $\left( M_l ,\omega_l ,p_l \right) $ be a sequence of pointed K\"ahler manifolds satisfying that $\Ric \left( \omega_l \right) \geq -\Lambda \omega_l $ and $\Vol \left( B_1 \left( p_l \right) \right) \geq v $, $\forall l\in\mathbb{N} $, where $\Lambda ,v>0$ are constants. Suppose that the compact balls $\bar{B}_{2} \left( p_l \right) $ converge to a metric space $(X,d,p)$ in the Gromov-Hausdorff sense, and the complex structures on $B_1 \left( p_l \right) $ converge to a complex manifold structure on $B_1 \left( p \right) $.  

Then there exists a sequence of integers $l_j \to\infty $, such that the K\"ahler potentials on $B_1 \left( p_{l_j} \right) $ converge to a K\"ahler space structure on $B_1 \left( p \right) $.
\end{lmm}

\begin{proof}
Let $x\in B_1 \left( p \right) $. By the convergence of complex structures, we can find a constant $r>0 $, a sequence of points $p_l \to p$, and a sequence of holomorphic charts on $B_r \left( p_l \right) $ so that the charts converge to a holomorphic chart on $B_r \left( p \right) $. Combining the Poincar\'e lemma with the H\"ormander's $L^2$ estimate, we can conclude that there are K\"ahler potentials $\phi_l $ on $B_r \left( p_l \right) $ by shrinking the value of $r$ if necessary. See also \cite[Lemma 2.8]{dns1}. Then $\left( B_r \left( p_l \right) ,\omega_{l} ,\mathbb{C} ,\phi_l \right) $ is a sequence of polarized K\"ahler manifolds.

Analysis similar to that in the proof of Proposition \ref{constructionlinebundleprop} shows that we can construct the limit K\"ahler space structure on $B_r (x)$ by choosing a suitable sequence $l_j \to \infty $. Choosing a locally finite open covering $\left\lbrace B_{r_j} \left( x_j \right) \right\rbrace $ of $B_1 (p)$. By taking a subsequence if necessary, the same argument as above gives the construction of the K\"ahler space structure on $B_1 (p)$, and the proof is complete.
\end{proof}

Now we prove that the K\"ahler space structure obtained in the above lemma does not depend on the choices of subsequence.

\begin{prop}
\label{proplocaluniquekahlerspace}
Let $\lbrace ( U_j ,\psi_j ) \rbrace $ and $\lbrace ( \widehat{U}_k ,\widehat{\psi}_k ) \rbrace $ be K\"ahler space structures constructed in Lemma \ref{lmmexistkahlerspace} as the limits of K\"ahler potentials in the Gromov-Hausdorff sense. Then they are equivalent.
\end{prop}

\begin{proof}
Without loss of generality, we can assume that the two K\"ahler space structures $\lbrace ( U_j ,\psi_j ) \rbrace $ and $\lbrace ( \widehat{U}_k ,\widehat{\psi}_k ) \rbrace $ are the limits of K\"ahler potentials on the sequences $M_{l_a}$ and $M_{l_b}$, respectively.

For each $x\in B_1 (p) $, we can find a constant $r>0$, sequences of points $x_{l_a} , x_{l_b } \to x $, and sequences of K\"ahler potentials and coordinates on $B_r \left( x_{l_a} \right) $, $B_r \left( x_{l_b} \right) $, so that the coordinates converge to a coordinate on $B_r (x)$, and the K\"ahler potentials converge to $\psi $ and $\widehat{\psi}$ on $B_r (x)$, respectively. So we can use these charts to identify the balls $B_{r'} \left( x_{l_a} \right) $ and $B_{r'} \left( x_{l_b} \right) $ with corresponding subsets of $B_r (x)$ for radius $r'<r$ and sufficiently large constants $l_a$, $l_b$. It is sufficient to show that $\psi -\widehat{\psi} $ is pluriharmonic on $B_r (x) $.

We will denote by $\Sigma_\epsilon $ the set of points $y\in X$ satisfying that $\lim_{t\to 0^+ } t^{-2n} \Vol \left( B_t (y) \right) \leq \V_{2n } -\epsilon $, $\forall \epsilon >0$, where $\V_{2n } $ is the volume of the unit ball in $\mathbb{C}^n $. By the relative volume comparison estimate, $\Sigma_\epsilon $ are closed subsets in $X$. Let us first show that $\psi -\widehat{\psi} $ is pluriharmonic on $B_r (x) \sq \Sigma_\epsilon $. 

Now we consider a smooth $(n-1,n-1)$-form $\alpha $ with compact support in $B_r (x) \sq \Sigma_\epsilon $. Our goal is to prove that $ \int_{B_r (x)} \partial\partialbar ( \psi -\widehat{\psi} )  \wedge \alpha =0 $. Let $\rho < \frac{1}{10} \dist ( \mathrm{supp}(\alpha ) ,\Sigma_\epsilon \cup \partial B_{r} (x) ) $ be a positive constant, where $\mathrm{supp}(\alpha ) $ is the support of $\alpha $. Then we can find a constant $\tau \in (0, \rho ) $, such that $d_{GH} \left( B_\delta (y) , B^{\mathbb{C}^n }_\delta (0) \right) \leq \Psi (\epsilon |n,\Lambda ,v )  $, $ \forall \delta \in (0,\tau ) $, $\forall y\in B_{r-\rho } (x) \sq B_\rho \left( \Sigma_\epsilon \right) $. It follows that we can find plurisubharmonic functions $\psi_{y,\delta ,\epsilon } $ and $\widehat{\psi}_{y,\delta ,\epsilon } $ on $B_\delta (y)$, such that 
$$\left| \psi_{y,\delta ,\epsilon } -\frac{d_y^2}{2} \right| + \left| \widehat{\psi}_{y,\delta ,\epsilon } -\frac{d_y^2}{2} \right| \leq \Psi \left( \epsilon |n,\Lambda ,v \right)\delta^2 , \;\; \forall \delta \in (0,\tau ),\; \forall y\in B_{r-\rho } (x) \sq B_\rho \left( \Sigma_\epsilon \right) ,$$
and $\psi_{y,\delta ,\epsilon} -\psi $, $\widehat{\psi}_{y,\delta ,\epsilon} -\widehat{\psi} $ are pluriharmonic functions, where $d_y $ is the distance from $y$. Let $y_{l_a} \in M_{l_a} $, $y_{l_b} \in M_{l_b}$ be sequences converge to $y$ in the Gromov-Hausdorff sense. Then we can find sequences of local K\"ahler potentials $ \psi_{y,\delta ,\epsilon ,l_a} $ and $ \widehat{\psi}_{y,\delta ,\epsilon ,l_b} $ on $ B_\delta ( y_{l_a} ) $ and $ B_\delta ( y_{l_b} ) $, respectively, such that $ \psi_{y,\delta ,\epsilon ,l_a} \to \psi_{y,\delta ,\epsilon } $ and $\widehat{\psi}_{y,\delta ,\epsilon ,l_b} \to \widehat{\psi}_{y,\delta ,\epsilon } $ in the Gromov-Hausdorff sense.

Choose a cut-off function $\eta :\mathbb{R} \to \mathbb{R} $ such that $\eta (t)=1 $ for $t\leq \frac{1}{2} $, $\eta (t)=0 $ for $t\geq 1 $, $-5\leq -\eta' \leq 0 $ and $\left| \eta'' \right| \leq 10$. Let $ \eta_{y,\delta ,\epsilon , l_a} (u) = \eta \left( 2\delta^{-2} \psi_{y,\delta ,\epsilon ,l_a} (u) \right) $. It is clear that $ \eta_{y,\delta ,\epsilon , l_a} \in C^{\infty}_{c} \left( B_{\delta } (y) \right) $. Note that we identify the ball $B_{r-\rho } \left( x_{l_a} \right) $ with an open subset of $B_r (x)$ for sufficiently large $l_a$. By the gradient estimate, $\Delta \psi_{y,\delta ,\epsilon ,l_a} =n $ now shows that $\left| \nabla \eta_{y,\delta ,\epsilon ,l_a} \right|^2 +\left| \Delta \eta_{y,\delta ,\epsilon ,l_a} \right| \leq C\delta^{-2} $ on $M_{l_a}$, where $C =C( n,\Lambda ,v ) >0 $ is a constant. 

For simplicity of notation, we use the same latter $C$ for large constants depending only on $n,\Lambda ,v $.

Now we can choose positive constants $\epsilon' < \epsilon $, $ \rho' < \rho $ and $\tau'<\tau $, such that $d_{GH} \left( B_\delta (y) , B^{\mathbb{C}^n }_\delta (0) \right) \leq \Psi (\epsilon' |n,\Lambda ,v )  $, $ \forall \delta \in (0,\tau' ) $, $\forall y\in B_{r-\rho } (x) \sq B_{\rho'} \left( \Sigma_{\epsilon'} \right) $. Suppose that $\delta <\tau' $ from now. Let $ \left\lbrace y_i \right\rbrace $ be a discrete set of points in $B_{r-\rho } (x) \sq B_\rho \left( \Sigma_\epsilon \right)$ satisfying that the balls $B_{\frac{\delta}{40}} (y_i)$ are disjoint, but $B_{r-\rho } (x) \sq B_\rho \left( \Sigma_\epsilon \right) \subset \cup_{i} B_{\frac{\delta}{5}} (y_i) $. It is clear that $\sum_{i} \eta_{y_i,\delta ,\epsilon ,l_a} \geq 1 $ on $B_{r-\rho } (x) \sq B_\rho \left( \Sigma_\epsilon \right)  $. By the relative volume comparison, we see that for every point $y\in B_{r-\rho } (x) \sq B_\rho \left( \Sigma_\epsilon \right) $, there are at most $C$ functions $\eta_{y_i ,\delta ,\epsilon ,l_a} $ that are not $0$ at $y$. Set $\eta'_{y_i ,\delta ,\epsilon ,l_a} = \frac{ \eta_{y_i,\delta ,\epsilon ,l_a}}{\sum_{\beta} \eta_{y_\beta ,\delta ,\epsilon ,l_a}} $. Then we have 
$$\left| \eta'_{y_i ,\delta ,\epsilon ,l_a} \right|_{\omega_{l_a}} + \delta \left| \nabla \eta'_{y_i ,\delta ,\epsilon ,l_a} \right|_{\omega_{l_a}} +\left| \partial\partialbar \eta'_{y_i ,\delta ,\epsilon ,l_a} \right|_{\omega_{l_a}} \leq C $$ 
for sufficiently large $l_a $. Similarly, we see that 
$$\left| \alpha \right|_{\omega_{l_a}} +\left| \partial \alpha \right|_{\omega_{l_a}} +\left| \partialbar\alpha \right|_{\omega_{l_a}} +\left| \partial\partialbar \alpha \right|_{\omega_{l_a}} \leq C $$ 
for sufficiently large $l_a $. Now we begin to estimate the following integral
\begin{eqnarray*}
 \left| \int_{B_r (x)} \partial\partialbar ( \psi -\widehat{\psi } ) \wedge \alpha \right| & = & \left| \sum_i \int_{B_\delta \left( y_i \right) }  \eta'_{y_i ,\delta ,\epsilon ,l_a} \partial\partialbar ( \psi_{y_i ,\delta ,\epsilon} -\widehat{\psi }_{y_i ,\delta ,\epsilon} ) \wedge \alpha \right| \\
& \leq & \sum_i \int_{B_\delta \left( y_i \right) } \left| ( \psi_{y_i ,\delta ,\epsilon} -\widehat{\psi }_{y_i ,\delta ,\epsilon} ) \partial\partialbar \left( \eta'_{y_i ,\delta ,\epsilon ,l_a} \wedge \alpha \right) \right| .
\end{eqnarray*}
For each $i$, we see that
$$\partial\partialbar \left( \eta'_{y_i ,\delta ,\epsilon ,l_a} \wedge \alpha \right) = \partial\partialbar \eta'_{y_i ,\delta ,\epsilon ,l_a} \wedge \alpha + \eta'_{y_i ,\delta ,\epsilon ,l_a} \partial\partialbar \alpha + \partial \eta'_{y_i ,\delta ,\epsilon ,l_a} \wedge \partialbar \alpha + \partial \alpha \wedge \partialbar \eta'_{y_i ,\delta ,\epsilon ,l_a} . $$
It follows that
\begin{eqnarray*}
 \left| \int_{B_r (x)} \partial\partialbar ( \psi -\widehat{\psi } ) \wedge \alpha \right| & \leq & \sum_i \int_{B_\delta \left( y_i \right) } \left| ( \psi_{y_i ,\delta ,\epsilon} -\widehat{\psi }_{y_i ,\delta ,\epsilon} ) \partial\partialbar \eta'_{y_i ,\delta ,\epsilon ,l_a} \wedge \alpha \right| \\
& & + \sum_i \int_{B_\delta \left( y_i \right) } C\delta \left| ( \psi_{y_i ,\delta ,\epsilon} -\widehat{\psi }_{y_i ,\delta ,\epsilon} )   \right| \\
& \leq & \sum_i \left( \int_{B_\delta \left( y_i \right) } \left| ( \psi_{y_i ,\delta ,\epsilon} -\widehat{\psi }_{y_i ,\delta ,\epsilon} ) \partial\partialbar \eta'_{y_i ,\delta ,\epsilon ,l_a} \wedge \alpha \right| +C\delta \right) .
\end{eqnarray*}
By the above, we have
$$ \int_{B_\delta \left( y_i \right) } \left| ( \psi_{y_i ,\delta ,\epsilon} -\widehat{\psi }_{y_i ,\delta ,\epsilon} ) \partial\partialbar \eta'_{y_i ,\delta ,\epsilon ,l_a} \wedge \alpha \right| \leq \Psi (\epsilon |n,\Lambda ,v) \delta^{2n} ,$$
for each $i$ and sufficiently large $l_a $. Moreover, when $y_i \in B_{r-\rho } \sq B_{\rho'} \left( \Sigma_{\epsilon'} \right) $, we have
$$ \int_{B_\delta \left( y_i \right) } \left| ( \psi_{y_i ,\delta ,\epsilon} -\widehat{\psi }_{y_i ,\delta ,\epsilon} ) \partial\partialbar \eta'_{y_i ,\delta ,\epsilon ,l_a} \wedge \alpha \right| \leq \Psi (\epsilon' |n,\Lambda ,v) \delta^{2n} .$$
By Cheeger-Jiang-Naber's estimate, we have $\Vol \left( B_{\rho'} \left( \Sigma_{\epsilon'} \right) \right) \leq C' \rho'^{2} $, where $C' = C' (n,\Lambda ,v,\epsilon' ) >0 $ is a constant. Hence we have $\# \left\lbrace i \big| y_i \in B_{\rho'} \left( \Sigma_{\epsilon'} \right) \right\rbrace  \leq C C' \rho'^{2} \delta^{-2n} .$ It follows that
\begin{eqnarray*}
 \left| \int_{B_r (x)} \partial\partialbar ( \psi -\widehat{\psi } ) \wedge \alpha \right| & \leq & \sum_{y_i \in B_{\rho'} \left( \Sigma_{\epsilon'} \right) } C \delta^{2n} + \sum_{y_i \notin B_{\rho'} \left( \Sigma_{\epsilon'} \right) } \Psi (\epsilon' |n,\Lambda ,v) \delta^{2n} \\
& \leq & C C' \rho'^2 + C \Psi (\epsilon' |n,\Lambda ,v) .
\end{eqnarray*}
Note that $\{ y_i \}_i \subset B_{r-\rho } (x) $. Choosing a sufficiently small constant $\rho' > 0$, and then letting $\epsilon' \to 0 $. Hence we can conclude that $\int_{B_r (x)} \partial\partialbar ( \psi -\widehat{\psi } ) \wedge \alpha =0 $. Now we see that the function $ \psi -\widehat{\psi } $ is pluriharmonic on $B_r (x) \sq \Sigma_\epsilon $.

Recall that $\Sigma_\epsilon $ is contained in a finite union of analytic subvarieties of $B_r (x) $ \cite[Theorem 4.1]{lggs1}. Since $\psi -\widehat{\psi} $ is continuous, it is pluriharmonic on $B_r (x)$. See Demailly \cite[Theorem 5.24]{dm1}. This completes the proof.
\end{proof}

We are now ready to prove Theorem \ref{thmuniquenesskahlerpotential}.

\vspace{0.2cm}

\noindent \textbf{Proof of Theorem \ref{thmuniquenesskahlerpotential}: }
The existence and uniqueness of the K\"ahler space structure on $X_{reg}$ have essentially been proved in Lemma \ref{lmmexistkahlerspace} and Proposition \ref{proplocaluniquekahlerspace} locally. The global case is an immediate consequence.

Note that $M_l$ are polarized. In this case, Proposition \ref{constructionlinebundleprop} shows that the K\"ahler space structure of $X_{reg}$ can be extended to $X$, and Proposition \ref{proplocaluniquekahlerspace} implies that the restriction of K\"ahler space structure on $X_{reg}$ is unique. Let $\lbrace ( U_j ,\psi_j ) \rbrace $ and $\lbrace ( \widehat{U}_k ,\widehat{\psi}_k ) \rbrace $ be two K\"ahler space structures on $X$, and $ \psi_j - \widehat{\psi }_k $ is pluriharmonic on $U_j \cap \widehat{U}_k \cap X_{reg} $ if $U_j \cap \widehat{U}_k \neq \emptyset $. Since $X$ is a normal space \cite[Corollary II-7.8]{dm1}, it is locally irreducible. Then the continuity of $ \psi_j - \widehat{\psi }_k $ implies that $ \psi_j - \widehat{\psi }_k $ is pluriharmonic on $U_j \cap \widehat{U}_k $ \cite[Theorem 1.7]{dm2}. Hence $\lbrace ( U_j ,\psi_j ) \rbrace $ and $\lbrace ( \widehat{U}_k ,\widehat{\psi}_k ) \rbrace $ are equivalent.
\qed

\vspace{0.2cm}

As a corollary, we have the following uniqueness property of limit Hermitian metric on a given limit line bundle.

\begin{coro}
\label{corouniquehermitianmetric}
Let $\left( \linebundle_\infty , h_\infty \right) $ and $\left( \linebundle_\infty , h'_\infty \right)$ be two limit line bundles as in Proposition \ref{constructionlinebundleprop}. Then there exists a pluriharmonic function $\psi $ on $X$, such that $h'_\infty =e^\psi h_\infty $.

Moreover, if $X$ is compact or $H_1 (X ;\mathbb{C} ) \cong 0 $, we have $\left( \linebundle_\infty , h_\infty \right) \cong \left( \linebundle_\infty , h'_\infty \right)$.
\end{coro}

\begin{proof}
Let $ \left\lbrace U_i \right\rbrace $ be an open covering of $X_{reg}$ such that the restriction of $ \linebundle_\infty $ on $U_i $ are trivial bundles. Then we can choose local frames $e_i $ of $ \linebundle_\infty $ on $U_i $. Recall that $-\frac{\sqrt{-1}}{2\pi } \log \left( h_\infty \right) $ and $-\frac{\sqrt{-1}}{2\pi } \log \left( h'_\infty \right) $ are limits of K\"ahler potentials locally. Hence we have a family of pluriharmonic functions $\psi_i $ on $U_i $, such that $h'_\infty ( e_i ) = e^{\psi_i } h_\infty ( e_i ) $. Since $h_\infty $ and $h'_\infty $ are Hermitian metrics on $ \linebundle_\infty $, we can conclude that $\psi_i = \psi_{i'} $ on $U_i \cap U_{i'} $ on each $U_i \cap U_{i'} \neq \emptyset $. It follows that $ \left\lbrace \psi_i \right\rbrace $ gives a pluriharmonic function $\psi $ on $X_{reg}$. By the construction in Proposition \ref{constructionlinebundleprop}, we see that $\psi $ can be extended to a continuous function on $X$, and hence $\psi $ can be extended to a pluriharmonic function on $X$. Note that $X$ is a normal complex space. Then we have $h'_\infty =e^\psi h_\infty $.

It is sufficient to prove that $\psi $ is the real part of a certain holomorphic function on $X$ when $X$ is compact or $H_1 (X ;\mathbb{C} ) \cong 0 $ now. If $X$ is compact, then the strong maximal principle shows that $\psi $ is constant, and hence we have $\left( \linebundle_\infty , h_\infty \right) \cong \left( \linebundle_\infty , h'_\infty \right)$ in this case.

Let $x\in X$. Our task now is to find a holomorphic function $\varphi_x $ around $x $ such that $ \mathrm{Re} ( \varphi_x ) =\psi $. Applying the resolution of singularities to $X$ \cite{hironaka1, bm1}, we can find an open neighborhood $U$ of $x$, a complex manifold $\widetilde{U} $, and a proper holomorphic surjective mapping $\pi_U : \widetilde{U} \to U $, such that the restriction of $\pi_U $ on $ \widetilde{U} \sq D $ gives a biholomorphic map $ \widetilde{U} \sq D \to U\cap X_{reg} $, where $D = \pi^{-1}_U (U\cap X_{sing} )  $ is the inverse image of $U\cap X_{sing} $. Set $\widetilde{\psi } = \psi \circ \pi_U $. Then $\widetilde{\psi } $ is pluriharmonic on $\widetilde{U} $. Clearly, $\widetilde{U} $ is pathwise connected. By the method of analytic continuation, we only need to show that the end value $\phi_\gamma (1 ) $ of the solution $\phi_\gamma (t ) $ of the equation $d\phi =-\sqrt{-1} \partial \psi  +\sqrt{-1} \partialbar \psi $, $\phi (\gamma (0) ) =0 $ along piecewise differentiable curves $\gamma : [0,1 ] \to \widetilde{U} $ depends only on the endpoint $\gamma (1 )$. Note that we can shrink the open neighborhood $U$ if necessary. It is easy to check that $\phi_\gamma (1 ) $ is unchanged under deformations of the path, so we can define $\phi_\gamma (1 ) $ for continuous curves.

Consider $A= \pi_U^{-1 } (x) $. Clearly, $A$ is a compact subvariety of $\widetilde{U} $, and hence $\widetilde{\psi } $ is constant on $A$. For each $y\in A$, we can find a small open neighborhood $V_y $ of $y$ such that $A\cap V_y $ is connected. Without loss of generality, we can assume that $V_y $ is contained in an open subset of $\widetilde{U} $ that is biholomorphic to the unit ball in $\mathbb{C}^n $. It follows that for each curve $\gamma_1 $ in $V_y $ connected two points in $A\cap V_y $, we can find a fixed endpoint homotopy between $\gamma_1 $ with a curve in $A$. Choosing a finite open covering $\{ V_{j } \}_j $ of $A$, satisfying that $V_j \subset V_{y_j} $ for some $y_j \in A$, and $V_j \cap V_k $ are connected.

Let $V=\bigcup_j V_j $. Construction similar to that in the proof of van Kampen's theorem \cite{aha1} now shows that any loop in $V$ can be homotopy to a loop in $A$. Since $\widetilde{\psi } $ is constant on $A$, we see that $\phi_\gamma (1 ) =0 $ for each curve $\gamma : [0,1 ] \to A $. It follows that $\phi_\gamma (1 ) =0 $ depends only on the endpoint $\gamma (1 )$, and hence there exists a holomorphic function $\varphi_A $ on $V$, such that $ \mathrm{Re} ( \varphi_A ) =\widetilde{\psi } $. Since $\pi_U $ is proper, we can find an open neighborhood $U' \subset U $ of $x$ such that $\pi^{-1}_U (U' ) \subset V $. Then $ \varphi_x = \varphi_A \circ \pi_U^{-1}$ gives a holomorphic function on $U'\cap X_{reg} $ satisfying $ \mathrm{Re} ( \varphi_x ) = \psi  $. Since $X$ is normal, $\varphi_x $ can be extended to a holomorphic function on $U' $. By the universal coefficient theorem, $H_1 (X ;\mathbb{C} ) \cong 0 $ if and only if $H_1 (X ;\mathbb{Z} ) $ is a torsion group. Hence all homomorphism $\pi_1 (X) \to\mathbb{C} $ must be trivial. Then we can construct a holomorphic function $\varphi $ on $X $ such that $ \mathrm{Re} ( \varphi ) =\psi $ by the method of analytic continuation again. This is our claim.
\end{proof}

\begin{rmk}
Let $X\cong \mathbb{D}^* $ be the punctured disc. Then $H_1 (X ;\mathbb{C} ) \cong \mathbb{C} $, and $\log |z| $ is pluriharmonic on $X$. But there is no holomorphic function $\varphi $ on $X $ such that $ \mathrm{Re} ( \varphi ) =\log |z| $.
\end{rmk}

\section{Convergence of Bergman kernels }
\label{convergencebergmankernelsection}

Let $\left( M_l ,\omega_l ,p_l \right) $ be a sequence of pointed complete polarized K\"ahler manifolds satisfying that $\Ric \left( \omega_l \right) \geq -\Lambda \omega_l $ and $\Vol \left( B_1 \left( p_l \right) \right) \geq v $, $\forall l\in\mathbb{N} $, where $\Lambda ,v>0$ are constants. Suppose that $ \left( M_l ,\omega_l ,p_l \right) \to \left( X ,d ,p \right) $ in the pointed Gromov-Hausdorff topology, and the normal complex space structure of $X$ is the limit of the sequence of complex structures of $M_l$. Then the upper semicontinuity of Bergman kernels follows from the definition.

\begin{lmm}
\label{lmmuppercontinuitybergman}
Let $(\linebundle_l ,h_l) $ be a sequence of Hermitian holomorphic line bundles $(\linebundle_l ,h_l)$ on $M_l$. Assume that the line bundles $(\linebundle_l ,h_l) $ converge to a Hermitian line bundle $(\linebundle_\infty ,h_\infty ) $ on $X_{reg}$. Then for each sequence of points $x_l \to x_{\infty } \in X_{reg} $, we have
$$ \limsup_{l\to\infty } \rho_{\linebundle_l , h_l ,1 } (x_l ) \leq \rho_{\linebundle_\infty , h_\infty ,1 } (x_\infty ) . $$
\end{lmm}

\begin{proof}
Let $B_{r_i} (y_i) $ be a locally finite covering of $X_{reg}$ satisfying that $B_{\frac{r_i}{100}} (y_i )$ is also a covering of $X_{reg}$, and for each $i$, there exist a sequence of points $y_{i,l} \to y_i $ and a sequence of charts $F_{i,l}$ on $B_{r_i} (y_{i,l}) \subset M_{l}$ for sufficiently large $l$, such that $F_{i,l}$ converge to a chart on $B_{r_i} (y_i) $. Since $(\linebundle_l ,h_l) $ converge to $(\linebundle_\infty ,h_\infty ) $, we can assume that there are a family of local frames $e_{i,l} \in H^0 \left( B_{r_i} (y_{i,l} ) ,\linebundle \right) $ and sections $ e_{i} \in H^0 \left( B_{r_i} (y_{i} ) ,\linebundle \right) $ such that the norms $ \left\Vert e_{i,l} \right\Vert $ and transition functions $f_{i,j,l} = e_{i,l} e^{-1}_{j,l}$ converge to the norms $ \left\Vert e_{i} \right\Vert $ and transition functions $f_{i,j} = e_{i} e^{-1}_{j}$, respectively.

Without loss of generality, we can assume that $ \limsup_{l\to\infty } \rho_{\linebundle_l , h_l ,1 } (x_l ) =b >0 $. Hence we can find a sequence of holomorphic sections $s_{l_k} \in H_{L^2}^0 \left( M_l ,\linebundle_l \right)  $ such that $l_k \to\infty $, $ \left\Vert s_{l_k} \right\Vert_{L^2} =1 $ and $ \lim_{k\to\infty} \left\Vert s_{l_k} \left( x_{l_k} \right) \right\Vert^2 = b $. Set $s_{l_k} = \phi_{i,l_k} e_{i,l_k } $ on $B_{r_i} \left( y_{i,l_k} \right) $. By taking a subsequence, the sequence of holomorphic functions $\phi_{i,l_k} $ converges to a holomorphic function $\phi_{i}$ on $B_{\frac{r_i}{100}} (y_i )$. The family of local holomorphic sections $ \left\lbrace \phi_{i} e_i \right\rbrace $ gives a holomorphic section $s \in H^0 \left( X_{reg} ,\linebundle \right) $ such that $ \left\Vert s \left( x_\infty \right) \right\Vert^2 = b $. By the volume convergence theorem \cite[Theorem 5.9]{chco3}, we have 
$$ \left\Vert s \right\Vert^2_{L^2} \leq \sup_{K\Subset X_{reg}} \lim_{l_k \to \infty } \int_{K} \left\Vert s_{l_k} \right\Vert^2 d\V_{\omega_{l_k }} \leq 1 .$$
It follows that $\rho_{\linebundle_\infty , h_\infty ,1 } (x_\infty ) \geq b $.
\end{proof}

Suppose that for each $x\in X$, there exist an open neighborhood $U$ of $x$ and an integer $D\in\mathbb{N} $, such that $\left( \linebundle^D ,h^D \right) $ can be extended to a Hermitian holomorphic line bundle on $U$. Since any holomorphic function on $X_{reg}$ can be extended to a holomorphic function on $X$, we see that $ \rho_{X, \mu ,\linebundle ,h,1} (x)$ can be defined as $  \sup_{ \left\Vert s \right\Vert_{L^2 } =1 }  \left( \left\Vert s^D (x) \right\Vert_{h^{D}}^2 \right)^{\frac{1}{D}}  ,$ where $D\in\mathbb{N}$ is an integer such that $\left( \linebundle^D ,h^D \right) $ can be extended to a Hermitian holomorphic line bundle on a neighborhood of $x\in X$. So we can extend $\rho_{X, \mu ,\linebundle ,h,m} $ to be a sequence of functions on $X$ in this case, $ \forall m\in\mathbb{N} $. It is easy to see that Lemma \ref{lmmuppercontinuitybergman} holds for all $x_\infty \in X $ in this case.

Before proving the lower semicontinuity of Bergman kernels, we recall the following estimate about the norms of holomorphic sections.

\begin{lmm}
\label{sobolevlmm}
Let $(M,\omega )$ be an $n$-dimensional K\"ahler manifold such that $\Ric(\omega ) \geq -\Lambda \omega $ for some $\Lambda \geq 0$, and $(\linebundle ,h)$ be a positive line bundle on $M$ equipped with a hermitian metric whose curvature form is $2\pi \omega $. Suppose that $B_2 (x)$ is relatively compact for some $x\in M$. Then there exists a constant $C=C(n,\Lambda )>0 $, such that for any $s\in H^0 (M,\linebundle )$, we have
\begin{eqnarray*}
 \Vert Ds (x) \Vert^2 + \Vert s (x) \Vert^2 & \leq &  \frac{C}{\Vol \left( B_1 (x) \right) } \int_{B_1 (x)} \Vert s \Vert^2 \omega^n  ,\;\; \forall x\in M.
\end{eqnarray*}
where $||\cdot ||$ denotes the norms associated with $h$ and $\omega $, and $D$ is the Chern connection on $\linebundle $.
\end{lmm}

\begin{proof}
By a straightforward computation, we obtain
\begin{eqnarray*}
\Delta \Vert s \Vert^2 & = & \Vert D s \Vert^2 -4n\pi \Vert s \Vert^2 , \\
\Delta \Vert D s \Vert^2 & = & \Vert DD s \Vert^2 -4\left(n+2 \right) \pi \Vert Ds \Vert^2 ,
\end{eqnarray*}
where $D$ is the covariant derivative of $(\linebundle ,h)$. Then the standard Moser iteration implies that
\begin{eqnarray*}
\sup_{B_{\frac{1}{2}} (x)} \Vert s \Vert^2 +\sup_{B_{\frac{1}{2}} (x)} \Vert Ds \Vert^2 & \leq & \frac{C}{ \Vol \left( B_1 (x) \right) } \int_{B_1 (x)} \Vert s \Vert^2 + \Vert Ds \Vert^2 \; \omega^n , 
\end{eqnarray*}
where $C$ is a constant that only depends on $n$ and $\Lambda $. See also \cite{gilt1} and \cite{hanlin1} for more details on Moser iteration. For notational convenience, the same latter $C$ will be used to denote constants depending on $n,\Lambda $. By Cheeger-Colding's construction in \cite{chco1}, there exists a $C^\infty $ function $\eta :M\to [0,1]$ such that $\eta =1$ on $B_{\frac{1}{2}} (x) $, $\eta =0$ on $M\sq B_1 (x) $, and $ |\nabla \eta | + |\Delta \eta | \leq C $. Then we have 
$$\int_{B_{1} (x)} \eta \Vert Ds \Vert^2 \; \omega^n = \int_{B_{1} (x)}  \Vert s \Vert^2  \Delta \eta + 4n\pi  \Vert s \Vert^2 \eta \; \omega^n \leq C  \int_{B_1 (x)} \Vert s \Vert^2 \; \omega^n , $$
and the lemma follows.
\end{proof}

Now we are ready to prove the lower semicontinuity of Bergman kernels. Our argument is to construct a sequence of holomorphic sections that approximates the given section. We can follow the arguments in \cite[Theorem 3.2]{donsun1}, \cite[Proposition 3.1]{lggs1}, and \cite[Lemma 5.3]{tg5}.  

\begin{lmm}
\label{lmmlowercontinuitybergman}
Let $(\linebundle_l ,h_l) $ be a sequence of Hermitian holomorphic line bundles $(\linebundle_l ,h_l)$ on $M_l$, such that $\Ric (h_l ) = 2\pi\omega_l $. Assume that the line bundles $(\linebundle_l ,h_l) $ converge to a Hermitian line bundle $(\linebundle_\infty ,h_\infty ) $ on $X_{reg}$. 

Then for each sequence of points $x_l \to x_{\infty } \in X $, we have
$$ \liminf_{l\to\infty } \rho_{\linebundle_l , h_l ,m } (x_l ) \geq \rho_{\linebundle_\infty , h_\infty ,m } (x_\infty ) ,\; \forall m>\frac{\Lambda}{2\pi} . $$
\end{lmm}

\begin{proof}
Without loss of generality, we can assume that $m=1$, and $\Lambda <2\pi $.

Let us first prove that the Bergman kernel of $\left( X , \mu , \linebundle_\infty ,h_\infty \right) $ at $x_\infty $ can be represented by $ \left\Vert s_{\infty} (x_\infty ) \right\Vert^2 $, where $\mu $ is the $2n$-dimensional Hausdorff measure on $X$, $s_{\infty} \in H^0_{L^2} \left( X_{reg},\linebundle_\infty \right) $ and $ \left\Vert s_{\infty} \right\Vert_{L^2} =1 $. Such a section $s_\infty $ is called a peak section at $x_\infty $. For each $t\in\mathbb{N}$, choosing a holomorphic section $s_{\infty ,t} \in H^0_{L^2} \left( X_{reg} ,\linebundle_\infty \right) $ satisfying that $ \left\Vert s_{\infty ,t} \right\Vert_{L^2} =1 $ and 
$$ \left\Vert s_{\infty ,t} (x_\infty ) \right\Vert^2 \geq \min\left\lbrace \rho_{\linebundle_\infty ,h_\infty ,1} (x_\infty ) -t^{-1} ,t  \right\rbrace .$$
Choosing a holomorphic chart $F = (z_1 ,\cdots ,z_n ) :B_{r} (y) \to U\subset\mathbb{C}^n $, where $r>0$ is a constant, and $y\in X_{reg}$. Then there exist a constant $c>0$ and a local frame $ e_y \in H^0 \left( B_{r} (y) ,\linebundle_\infty \right) $ such that $c \leq \Vert e_y \Vert \leq c^{-1} $ and $ \omega_\infty \geq c F^* \omega_{Euc} $. We can shrink the value of $r$ if necessary. Let $f_{\infty ,t,y} \in \mathcal{O} \left( B_r (y) \right) $ satisfying $ s_{\infty ,t} =f_{\infty ,t,y} e_y $. Hence we have a constant $C>0 $ independent of $t$, such that $ \int_U |f_{\infty ,t,y}|^2 d\V_{Euc} \leq C $, and the mean value equation of holomorphic functions now shows that $\sup_{K} |f_{\infty ,t,y}|^2 \leq C' $ for any compact $K\subset U$, where $C'>0 $ is also a constant independent of $t$. By taking a suitable subsequence, we can assume that $s_{\infty ,t}$ converges to a holomorphic section $s_{\infty} \in H^0_{L^2} \left( X_{reg} ,\linebundle_\infty \right) $ uniformly on each compact set $K\subset X_{reg} $. 

By the above, we have $ \left\Vert s_{\infty} \right\Vert_{L^2} \leq 1 $, and $ \rho_{\linebundle_\infty ,h_\infty ,1} (x_\infty ) = \left\Vert s_{\infty} (x_\infty ) \right\Vert^2 <\infty $ when $x_\infty \in X_{reg}$. Now we assume that $x_\infty \in X_{sing}$. Since $ \left( \linebundle_\infty^D ,h_\infty^D \right) $ can be extended to a Hermitian line bundle around $x_\infty $, we can find constants $r' ,c'>0$ and a local frame $ e_x \in H^0 \left( B_{r'} (x_\infty ) ,\linebundle^D_\infty \right) $ such that $c' \leq \Vert e_x \Vert \leq c'^{-1} $. Let $f_{\infty ,t,x} \in \mathcal{O} \left( B_{r'} (x_\infty ) \right) $ satisfying $ s^D_{\infty ,t} =f_{\infty ,t,x} e_x $. Then we have $ \int_{B_{r'} (x_\infty )} |f_{\infty ,t,x}|^2 d\mu \leq C'' $ for some constant $C''>0$. By the desingularization of analytic varieties \cite{bm1}, there exist a complex manifold $U'$ and a proper holomorphic map $\sigma : U'\to B_{r'} (x_\infty ) $ such that $E=\sigma^{-1} ( X_{sing} ) $ is contained in a normal crossing divisor in $U'$. By the Cauchy's integral formula, the inner closed uniform convergence of $ f_{\infty ,t,x} \circ \sigma $ on $U'\sq E$ now shows that $ f_{\infty ,t,x} \circ \sigma $ is inner closed uniform convergence on $U'$. Hence $ s^D_{\infty ,t} $ converges to a holomorphic section $s'_{\infty} \in H^0_{L^2} \left( X_{reg} ,\linebundle^D_\infty \right) $ uniformly around $x_\infty $. Clearly, $s_{\infty ,t}$ converge to a holomorphic section $s_{\infty} \in H^0_{L^2} \left( X_{reg} ,\linebundle_\infty \right) $ as $t\to\infty $. It follows that $s^D_\infty = s'_\infty $, and $ \rho_{\linebundle_\infty ,h_\infty ,1} (x_\infty ) = \left\Vert s_{\infty} (x_\infty ) \right\Vert^2 <\infty $.

The next thing to do in the proof is to find a sequence of holomorphic sections to approximate the peak section. Fix a small constant $\epsilon >0$. By the convergence of complex structures, there exist an open covering $ \{ B_r (y_i ) \}_{i=1}^N $ of $B_{\epsilon^{-1}} (x_\infty ) \sq B_\epsilon (X_{sing} ) $, sequences of points $ M_{l} \ni y_{i,l} \to y_i $ and sequences of charts $F_{i,l}$ on $B_{2r} (y_{i,l}) \subset M_{l}$ for sufficiently large $l$, such that $F_{i,l}$ converge to a chart on $B_{2r} (y_i) $. Since $(\linebundle_l ,h_l )$ converges to $( \linebundle_\infty ,h_\infty )$, we can assume that there are a family of local frames $e_{i,l} \in H^0 \left( B_{2r } (y_{i,l} ) ,\linebundle \right) $ and local frames $ e_{i} \in H^0 \left( B_{2r } (y_{i} ) ,\linebundle \right) $ such that the norms of frames $ \left\Vert e_{i,l} \right\Vert $ and transition functions $f_{i,j,l} = e_{i,l} e^{-1}_{j,l}$ converge to the norms $ \left\Vert e_{i} \right\Vert $ and transition functions $f_{i,j} = e_{i} e^{-1}_{j}$, respectively.

Let $s_\infty $ be the peak section at $x_\infty \in X $. Then we have $s_\infty =\psi_i e_i $ on $B_{2r} (y_i ) $ for each $i$, where $\psi_i \in\mathcal{O} (B_{2r} (y_i ) ) $. Choosing cut-off functions $\eta_i \in C_c^\infty ( F_i ( B_{2r} (y_i ) ) ) $ such that $ 0\leq \eta_i \leq 1 $, and $\eta_i =1 $ on $ F_i \left( B_{\frac{3r}{2}} (y_i ) \right) $. Hence $ \sum_i \eta_i \circ F_{i,l} \geq 1 $ on $\cup_{i} B_r (y_{i,l}) $ for sufficiently large $l$. Set $\eta_{i,l} =\frac{\eta_i \circ F_{i,l} }{\sum_j \eta_j \circ F_{j,l} } $. Then $s'_{l} = \sum_{i} \eta_{i,l} \cdot \left( \psi_i \circ F_{i,l} \right) e_{i,l} $ gives a smooth section on $\cup_{i} B_r (y_{i,l}) $ for sufficiently large $l$. Clearly, we have $\left\Vert  s'_l \right\Vert_{L^2} \leq 1+\Psi (l^{-1} | n,\epsilon ) $, $\left\Vert \partialbar s'_l \right\Vert_{L^2} \leq \Psi (l^{-1} | n,\epsilon ) $ and $\left\Vert  s'_l \right\Vert \leq C $ for sufficiently large $l$, where $C>0$ is a constant independent of $l$ and $\epsilon $. For abbreviation, we use the same letter $C$ for constants independent of $l$ and $\epsilon $. By approximating the function $ \chi_{\epsilon} (y) = \eta \left( 10^{-1} \epsilon^{-1} \dist (y,X_{sing}) \right) $, we can construct smooth function $\chi_{\epsilon ,l} $ on $B_{\epsilon^{-1}} (x_l ) $ for sufficiently large $l$, where $ \eta\in C^\infty (\mathbb{R}) $, $\eta'\geq 0$, $\eta (t)=0$ for $t\leq 0$, and $\eta (t) =1 $ for $t\geq 1$. Without loss of generality, we can assume that $\left| \nabla \chi_{\epsilon ,l} \right| \leq C \epsilon^{-1} $ and $\mathrm{supp} \chi_{\epsilon ,l } \subset \cup_{i} B_r (y_{i,l}) $. Let $\varphi_{\epsilon ,l} \in C^{\infty}_{c} (B_{\epsilon^{-1} } (x_l) )$ such that $\left| \nabla \varphi_{\epsilon ,l} \right| \leq C\epsilon$, $ \varphi_{\epsilon ,l} =1 $ on $B_{2^{-1} \epsilon^{-1}} (x_l ) $, and $ \varphi_{\epsilon ,l} =0 $ outside $B_{ \epsilon^{-1}} (x_l ) $. Hence we have $\left\Vert \partialbar \left( \varphi_{\epsilon ,l} \chi_{\epsilon ,l} s'_l \right) \right\Vert \leq C\left\Vert \partialbar s'_l \right\Vert +C\epsilon +C\left| \nabla \chi_{\epsilon ,l} \right| $. Recall that $y\in X_{reg}$ when there exists a tangent cone at $y\in X$ splitting off $\mathbb{R}^{2n-2}$. Then Theorem \ref{jcwsjan1thm17coro} shows that $ \left| \nabla \chi_{\epsilon ,l} \right|_{L^2} \leq C\epsilon $. It follows that $\left\Vert \partialbar \left( \varphi_{\epsilon ,l} \chi_{\epsilon ,l} s'_l \right) \right\Vert_{L^2} \leq \Psi (l^{-1} | n,\epsilon ) +C\epsilon $.

By H\"ormander's $L^2$-estimate, we can find a $\linebundle$-valued smooth function $\xi_l $ on $M_l $, satisfying that $\partialbar\xi_l = \partialbar \left( \varphi_{\epsilon ,l} \chi_{\epsilon ,l} s'_l \right) $ and $ \left\Vert \xi_l \right\Vert_{L^2} \leq \frac{2\pi}{2\pi -\Lambda} \left\Vert \partialbar \left( \varphi_{\epsilon ,l} \chi_{\epsilon ,l} s'_l \right) \right\Vert_{L^2} \leq \Psi (l^{-1} | n,\epsilon ) +C\epsilon $. Let $s_l =s'_l -\xi_l $. By choosing suitable constant $\epsilon$ for each $l$, the integral $\left\Vert \xi_l \right\Vert_{L^2} \leq \Psi ( l^{-1} ) $, and $\left\Vert s'_l \right\Vert $ converge to $\left\Vert s_\infty \right\Vert $. Note that $\left\Vert s'_l \right\Vert $ converge to $\left\Vert s_\infty \right\Vert $ on $B_{\epsilon^{-1}} (x_\infty ) \sq B_\epsilon (X_{sing} ) $. Now Lemma \ref{sobolevlmm} shows that $\left\Vert s_l (x_l) \right\Vert$ is very close to $\left\Vert s_\infty (x_\infty ) \right\Vert$ for sufficiently large $l$. Hence $\liminf_{l\to\infty } \rho_{\linebundle_l , h_l ,1 } (x_l ) \geq \rho_{\linebundle_\infty , h_\infty ,1 } (x_\infty ) ,$ and the proof is complete.
\end{proof}

Combining Proposition \ref{constructionlinebundleprop}, Lemma \ref{lmmlowercontinuitybergman} and Lemma \ref{lmmuppercontinuitybergman}, we can obtain Theorem \ref{thmcontinuousbergmankernel}. We rewrite it here.

\begin{thm}
Let $\left( \linebundle_l ,h_l \right) $ be a sequence of Hermitian line bundles on $M_l $ such that $\Ric \left( h_l \right) = 2\pi \omega_l $. Then we can find a sequence of integers $l_j \to\infty $, such that $\left( \linebundle_{l_j} , h_{l_j} \right) $ converges to a holomorphic line bundle $\left( \linebundle_{\infty } , h_{\infty } \right) $ on $X_{reg } $, and for each $x\in X$, there exists an integer $D\in\mathbb{N}$, such that $\left( \linebundle_\infty^D ,h_\infty^D \right) $ converges to a Hermitian holomorphic line bundle on a neighborhood of $x\in X$. And the Bergman kernels $\rho_{ M_{l_j} ,\mu , \linebundle_{l_j} , h_{l_j} , m} $ converge to the Bergman kernel $\rho_{X,\mu ,\linebundle_\infty ,h_\infty ,m} $ in the Gromov-Hausdorff sense, $\forall  m> \frac{\Lambda}{2\pi} $, where $\mu $ is the $2n$-dimensional Hausdorff measure.
\end{thm}

\begin{rmk}
By the same argument, we can show that the Bergman kernels $\rho_{m} (p_l ) $ converge to $\rho_{X,m} (p) $ when $M_l $ are pseudoconvex (need not be complete) and $B_{1} (p_l )$ are relatively compact in $M_l $.
\end{rmk}

We conclude this section by pointing out that Theorem \ref{thmcontinuousbergmankernel} may not hold when $m\leq \frac{\Lambda}{2\pi} $.

\begin{exap}
\label{genus2example}
Let $M$ be a complex curve with genus $2$. Then there exists a unique K\"ahler metric $\omega_M$ on $M$ such that $\int_M {\omega_M} =2 $ and the sectional curvature $ \equiv -2\pi $. Since the canonical bundle $K_M $ satisfying $\deg (K_M ) =2 $, we see that $[\omega_M ] \in c_1 (K_M ) $. Choosing a sequence of points $\{ p_l \}_{l=1}^{\infty} $ in $M$ so that $p_l \to p_\infty \in M $ and $p_l\neq p_\infty $, $\forall l\in\mathbb{N}$. Let $\linebundle_l = K_M +(p_l ) -(p_\infty ) $ for each $l$. Then there is a unique Hermitian metric $h_l $ on $\linebundle_l$ such that $\Ric (h_l) =2\pi\omega_M $. Clearly, we have $(\linebundle_l ,h_l ) \to (K_M ,h_\infty ) $, where $h_M$ is the unique Hermitian metric on $K_M$ satisfying $\Ric (h_\infty ) =2\pi\omega_M $. But the Riemann-Roch formula implies that
\begin{eqnarray*}
\int_M \left( \rho_{\linebundle_l ,h_l ,1 } - \rho_{K_M ,h_\infty ,1 } \right) \omega_M & = & \dim H^0 (M,\linebundle_l ) - \dim H^0 (M,K_M ) \\
& = & \dim H^0 (M,K_M \otimes \linebundle^{-1}_l ) - \dim H^0 (M,\mathcal{O} ) =-1 ,\;\forall l \in\mathbb{N} .
\end{eqnarray*}
\end{exap}

\section{\texorpdfstring{$L^2$}{Lg} estimate on Gromov-Hausdorff limits }
\label{l2estghlimitsection}

In this section, we will establish a version of $L^2$-estimate on the Gromov-Hausdorff limits.

Let $ (X,d,p)  $ be the pointed Gromov-Hausdorff limit of a sequence of pointed complete polarized K\"ahler manifolds $\left( M_l ,\omega_l ,p_l \right) $ with $\Ric \left( \omega_l \right) \geq -\Lambda \omega_l $ and $\Vol \left( B_1 \left( p_l \right) \right) \geq v $, $\forall l\in\mathbb{N} $, where $\Lambda ,v>0$ are constants. Then $X$ is a normal complex space, and we can assume that the complex structures of $M_l $ converge to the complex structure of $X$ by choosing a suitable subsequence. Theorem \ref{thmuniquenesskahlerpotential} shows that $\omega_l $ converge to a unique closed positive current $\omega_X $ on $X_{reg} $. The $L^2$ estimate we consider is about this current. Note that the inverses $g^{j\bar{k}}_i$ of metric tensors are convergence on $X_{reg}$ \cite[Lemma 3.4]{lggs2}, so the norms of smooth $(0,1)$-forms are well defined.

\begin{prop}
\label{propl2estimateghlimit}
Let $\linebundle'_\infty $ be a line bundle on $X$ with a continuous Hermitian metric $h'_\infty $. Assume that $\linebundle'_\infty $ is the limit line bundle of a sequence of line bundles on $\left( M_l ,\omega_l \right) $. Suppose that $ \varphi \in L^1_{loc} (X) $ can be approximated by a decreasing sequence of smooth function $\left\lbrace \varphi_i \right\rbrace_{i=1}^{\infty} $ on $X$ such that
\begin{equation*}
\label{l2conditionformula}
\sqrt{-1} \partial\partialbar \varphi_i + \Ric (h'_\infty ) - \Lambda \omega_X \geq c\omega_X ,\; \forall i\in\mathbb{N},
\end{equation*}
where $c>0$ is a constant. Then for any $\linebundle'_\infty $-valued smooth function $\xi \in L^2$ on $X$, there exists an $\linebundle'_\infty $-valued function $u\in L^{2}$ such that $\partialbar u= \alpha = \partialbar \xi $ and 
\begin{equation}
\label{l2propformula}
\int_{X} \Vert u \Vert^{2} e^{-\varphi } d\mu \leq \int_{X} c^{-1} \Vert \alpha \Vert^{2} e^{-\varphi } d\mu .
\end{equation}

If we further assume that the non-collapsing condition $\Vol (B_1 (x) ) \geq v $ holds for all $x\in X$, then this estimate (\ref{l2propformula}) holds for all smooth $\linebundle'_\infty $-valued $(0,1)$-form $\alpha \in L^2 $ satisfying $\partialbar \alpha =0 $.
\end{prop}

\begin{proof}
We divide our proof in three steps. Without loss of generality, we can assume that $\varphi =0 $. 

First, we need to regularize the singular Hermitian metric $ h'_\infty $. We divide it into two parts to approximate, using Demailly's regularization theorem \cite{bk1, dp1}, and approximating the other part by using K\"ahler potentials on $M_l$ as in the proof of Liu-Sz\`ekelihidi's version of $L^2$-estimate on tangent cone \cite[Proposition 3.1]{lggs2}.

Fix $R>0$. Choosing an open coverings $\{ B_r (y_i) \}_{i=1}^{N} $ of $B_R (p ) \subset X $, where $r>0$ is a constant. By the convergence of K\"ahler space structures, we can assume that there are sequences of points $ M_{l} \ni y_{i,l} \to y_i $ and sequences of K\"ahler potentials $\psi_{i,l}$ on $B_{2r} (y_{i,l}) \subset M_{l}$ for sufficiently large $l$, such that $|\psi_{i,l} | +|\nabla \psi_{i,l} | \leq 1 $, and $\psi_{i,l}$ converge to a K\"ahler potential $\psi_i$ on $B_{2r} (y_i) $, $i=1,\cdots ,N$. Since $\linebundle'_\infty $ is a limit line bundle, we can find a sequence of line bundles $\linebundle'_l $ on $M_l $ and a sequence of local frames $e_{i,l} \in H^0 \left( B_{2r} (y_{i,l} ) ,\linebundle'_l \right) $ such that the transition functions $f_{ij,l} = e_{i,l} e^{-1}_{j,l} $ converge to transition functions $f_{ij} = e_{i} e^{-1}_{j} \in \mathcal{O} \left( B_{2r} (y_i) \cap B_{2r} (y_j) \right) $ of $\linebundle'_\infty $, where $e_i \in H^0 \left( B_{2r} (y_i ) ,\linebundle'_\infty \right) $ is a local frame. Note that we can shrink the value of $r$ if necessary. Let $-\log h'_\infty \left( e_i ,e_i \right) =\phi_i $. Then $\rho_i = \phi_i - (\Lambda +c )\psi_i $ is a continuous plurisubharmonic function on $B_{2r} (y_i) $, $i=1,\cdots ,N$, and the difference $\rho_i -\rho_j $ is pluriharmonic on each $B_{2r} (y_i) \cap B_{2r} (y_j) \neq\emptyset $. 

By the convergence of complex structures, by shrinking the radius $r>0$, we can assume that there exists a sequence of holomorphic maps $F_{i,l} = (z_1 ,\cdots ,z_{N_i} ) : B_{2r} (y_{i,l}) \to \mathbb{C}^{N_i } $ converge to an injective holomorphic map $F_i = (z_1 ,\cdots ,z_{N_i} ) : B_{2r} (y_i) \to \mathbb{C}^{N_i } $ for each $i$. Let $\delta >0$ be a small constant. Then $ U_{i,x,t} = F_i^{-1} \left( B_{t\delta} \left( F_i (x) \right) \right) \subset B_{2r} (y_i ) $ are Stein space, $\forall t\in (0,1]$, $\forall x\in B_{\frac{7r}{4}} (y_i ) $, $\forall i =1,\cdots ,N$. Similarly, set $ U_{i,x,t,l} = F_{i,l}^{-1} \left( B_{t\delta} \left( F_{i,l} (x) \right) \right) \subset B_{2r} (y_{i,l} ) $ for $ x\in B_{\frac{7r}{4}} (y_{i,l} ) $. 

By Forn{\oe}ss-Narasimhan \cite[Theorem 5.5]{fm1} we have smooth plurisubharmonic functions $\rho_{i,x} $ on a neighborhood of $F_{i} \left( U_{i,x,1} \right) $ such that $| \rho_{i,x} \circ F_i - \rho_i | \leq \delta^{10} $ on $U_{i,x,1}$. We can find $N'$ points $\{ x_k \}_{k=1}^{N'} \subset B_R (p) $ satisfying that $x_k \in B_{\frac{7r}{4}} (y_{i_k} ) $ for some $i_k $, and $B_R (p) \subset \cup_{k=1}^{N'} U_{i_k ,x_k ,\frac{1}{5} } $. Choosing a sequence of points $ M_l \ni x_{k,l} \to x_k $ for each $k$. Then we have $B_R (p_l ) \subset \cup_{k=1}^{N'} U_{i_k ,x_{k,l} ,\frac{1}{4} ,l } $ for sufficiently large $l$. On each $U_{i_k ,x_k ,1,l} \cap B_{\frac{3r}{2}} (y_{i,l} ) \neq\emptyset $, we define

\begin{eqnarray*}
\varrho_{i,k,l} (w) & = & \rho_{i_k ,x_k } \circ F_{i_k ,l} (w) -\delta^2 \left| F_{i_k ,l} (w) -F_{i_k ,l} ( x_k ) \right|^2 \\
& & - 2\log |f_{i,i_k ,l} (w) | +(\Lambda +c) \left( \psi_{i,l} (w) -\psi_{i_k ,l} (w) \right) ,
\end{eqnarray*} 
$\forall w\in U_{i_k ,x_k ,1,l} \cap B_{\frac{3r}{2}} (y_{i,l} ) .$  Now we construct functions $ \widetilde{\rho}_{i,l} $ on $B_{\frac{3r}{2}} (y_{i,l} )$ as following:
$$ \widetilde{\rho}_{i,l} (w) = \max_{ U_{i_k ,x_k ,1,l} \ni w } \varrho_{i,k,l} (w),\; \forall w\in B_{\frac{3r}{2}} (y_{i,l} ). $$
Clearly, we have 
$$ \varrho_{i,k,l} \to \varrho_{i,k,\infty } = \rho_{i_k ,x_k } \circ F_{i_k} -\delta^2 \left| F_{i_k } -F_{i_k } ( x_k ) \right|^2 +\rho_{i} -\rho_{i_k}   $$ 
as $l\to\infty $, and hence $| \rho_{i_k ,x_k } \circ F_i - \rho_{i_k} | \leq \delta^{10} $ implies that
$$\sup_{w \notin U_{i_k ,x_k ,\frac{2}{3},l} } \widetilde{\rho}_{i,l} (w) \leq \inf_{w \in U_{i_{k'} ,x_{k'} ,\frac{1}{3},l} } \widetilde{\rho}_{i,l} (w) -\frac{\delta^4}{20} ,\; \forall k,k' =1,\cdots ,N' , $$
for sufficiently large $l$. It follows that 
$$\sqrt{-1} \partial\partialbar \widetilde{\rho}_{i,l} \geq -\delta^2 \sum_{k} F^{*}_{i_k} \omega_{Euc} \geq -C \delta^2 \omega_l ,$$ 
where $C>0$ is a constant independent of $l$ and $\delta $. See \cite[Lemma I-5.17]{dm1}. When $B_{\frac{3r}{2}} (y_{i,l}) \cap B_{\frac{3r}{2}} (y_{i',l}) \cap U_{i_k,x_k,1,l} \neq\emptyset $, we have 
$$\varrho_{i,k,l} -\varrho_{i' ,k,l} = 2\log |f_{i',i,l}|^2 - ( \Lambda +c ) (\psi_{i',l} -\psi_{i,l} ) $$ 
on $B_{\frac{3r}{2}} (y_{i,l}) \cap B_{\frac{3r}{2}} (y_{i',l}) \cap U_{i_k,x_k,1,l} .$ 
Then 
$$\widetilde{\rho}_{i,l} -\widetilde{\rho}_{i',l} = \max_{ k } \varrho_{i,k,l} - \max_{ k } \varrho_{i' ,k,l} = 2\log |f_{i',i,l}|^2 - ( \Lambda +c ) (\psi_{i',l} -\psi_{i,l} ) ,$$ 
on $B_{\frac{3r}{2}} (y_{i,l}) \cap B_{\frac{3r}{2}} (y_{i',l}) $. If we use regularized max function \cite[Lemma I-5.18]{dm1} instead of maximum in the above, we can get smooth $ \widetilde{\rho}_{i,l} $.

A straightforward calculation shows that 
\begin{equation}
\label{l2constructsmoothmetricformula}
h'_{l,\delta } (e_{i,l} ,e_{i,l} ) = e^{-(\Lambda +c)\psi_{i,l} -\widetilde{\rho}_{i,l} } ,\;\forall i=1,\cdots ,N,
\end{equation}
gives a smooth Hermitian metric $h'_{l,\delta } $ of $\linebundle'_l $ on $\cup_{i} B_{\frac{3r}{2}} (y_{i,l})\supset B_R (p_l ) $, and $\Ric (h'_{l,\delta } ) \geq (\Lambda +c - C\delta^2 )\; \omega_l . $ By choosing a suitable sequence of radii $R$ and constants $\delta $, we can construct a sequence of smooth Hermitian metric $\widetilde{ h}'_{l} $ of $\linebundle'_l $ on $ B_{R_l} (p_l ) $ such that $\lim_{l\to\infty} R_l =\infty $, $\Ric ( \widetilde{ h}'_{l} )\geq (\Lambda +c -\Psi (l^{-1} ) )\; \omega_l , $ and $\widetilde{ h}'_{l} \to h'_\infty $ as $l\to\infty $.
 
Another step in the proof is to establish the $L^2 $ estimate for $\partialbar\xi $. We can follow the argument in \cite[Proposition 3.1]{lggs2} here. This step can be reduced to prove that for any smooth $\linebundle'^{-1}_\infty $-valued $(n,n-1)$-form $\beta $ with compact support in $X_{reg} $, we have
\begin{equation}
\label{l2step2formula}
\left( \int_{X} \left\langle \xi ,\star \partialbar\beta \right\rangle d\mu \right)^2 = \left( \int_{X} \xi \wedge \partialbar\beta \right)^2 \leq c^{-1} \left( \int_X \left\Vert \alpha \right\Vert^2 d\mu \right) \left( \int_X \left\Vert \star \partialbar\beta \right\Vert^2 d\mu \right) .
\end{equation}

By abuse of notation, we will denote by $\{ B_{10 r_i} (y_i ) \}_{i=1}^{\infty}$ a locally finite open covering of $X_{reg}$. For each $i$, choose a sequence of points $ M_l \ni y_{i,l} \to y_i $. Without loss of generality, we can assume that $\{ B_{ r_i} (y_i ) \}_{i=1}^{\infty}$ is also an open covering of $X_{reg} $, and there are holomorphic charts $F_{i,l} : B_{10r_i} (y_{i,l} ) \to \mathbb{C}^n $ converging to holomorphic charts $F_{i } : B_{10r_i} (y_{i } ) \to \mathbb{C}^n $ as $l\to\infty $. Since $( \widetilde{\linebundle}'_{i,l} , \widetilde{h}'_{i,l} ) \to ( \linebundle'_\infty ,h'_{\infty} ) $, it follows that there are local frames $e'_{i,l} \in H^0 \left( B_{10r_i} (y_{i,l} ) ,\widetilde{\linebundle}'_{l} \right) $ such that the norms $\Vert e'_{i,l} \Vert $ and transition functions $ e'_{i,l} e'^{-1}_{i',l} $ converge to the norms $\Vert e'_{i } \Vert $ and transition functions $ e'_{i } e'^{-1}_{i' } $, where $e_{i } \in H^0 \left( B_{10r_i} (y_{i } ) , \linebundle' \right) $ are local frames. Note that we can choose smaller balls instead of $B_{10r_i} $ if necessary. Set $\xi =\xi_i e'_i $ on $B_{10r_i} (y_{i } ) $. Then $\xi_{i,l} = \xi_{i} \circ F_{i}^{-1} \circ F_{i,l} \in C^\infty \left( B_{10r_i} (y_{i,l} ) \right) $ for sufficiently large $l $. Similarly, we can define $\beta_{i,l} $. As in the proof of Lemma \ref{lmmlowercontinuitybergman}, we have cut-off functions $\eta_{i,l} \in C^\infty_c \left( B_{10r_i} (y_{i,l} ) \right) $, such that $\sum_{i} \eta_{i,l} =1 $ on $B_{9r_j} (y_{j,l} ) $ for any fixed $j$ and sufficiently large $l$, and $|\nabla \eta_{i,l} |\leq C $, where $C$ is a constant independent of $l$. For simplicity of notation, we use the same latter $C$ for large constants independent of $l $.

Fix a large radius $R'>10$ and a small constant $\epsilon >0 $. Hence we have $N_{\epsilon ,R'} $ points $\{ x'_{k } \} \subset X_{sing} $ such that $B_{100^{-1}\epsilon } (x'_{k}) $ are disjoint, and $ X_{sing} \cap B_{R' } (p) \subset \cup_k B_{50^{-1} \epsilon} (x'_k) $. Choosing sequences of points $M_l \ni x'_{k,l } \to x'_k $. Then we can construct cut-off functions $\chi_{R' ,l} \in C^\infty_c \left(  B_{ R' } (p_l ) \right) $ and $\eta_{\epsilon ,l } \in C^\infty_c \left( \cup_k B_{ \epsilon} (x'_{k,l} ) \right) $ for sufficiently large $l$, such that $\chi_{R' ,l} =1 $ on $B_{2^{-1} R'} (p_l ) $, $\eta_{\epsilon ,l} =1 $ on $\cup_k B_{2^{-1} \epsilon} (x'_{k,l} ) $, $|\nabla \chi_{R' ,1} | \leq CR'^{-1} $, and $|\nabla \eta_{\epsilon ,l} | \leq C\epsilon^{-1} $. Let $\xi_{l} = \sum_{i} \eta_{i,l} \xi_{i,l} e'_{i,l} $ and $\beta_{l} = \sum_{i} \eta_{i,l} \beta_{i,l} e'^{-1}_{i,l} $ on $B_{R'} (p_l ) \sq \cup_k B_{10^{-1} \epsilon} (x'_{k,l}) $. Since the support of $\beta $ is compact in $X_{reg}$, we can assume that $\mathrm{supp} \beta_{l} \Subset B_{10^{-1}R'} (p_l ) \sq \cup_k B_{10 \epsilon} (x'_{k,l}) $. Then we decompose $ (\star\beta_l )=\nu_l = \nu_{1,l} + \nu_{2,l}  $ on $B_{R'} (p_l ) $ under the $L^2$ orthogonal decomposition $( \ker \partialbar ) \bigoplus ( \ker \partialbar )^\perp $. By the Bochner-Kodaira-Nakano inequality \cite[Formula VII-(2.1)]{dm1}, we have $\star \partialbar\beta_l = -\partialbar^* \nu = -\partialbar \nu_{1,l} $, and hence
\begin{eqnarray*}
\left( \int_{B_{ R'} (p_l ) } \left\langle \xi_l ,\star \partialbar\beta_l \right\rangle d\mu \right)^2 & = & \left( \int_{B_{ R'} (p_l ) } \left\langle \partialbar \left( \chi_{R'} (1-\eta_{\epsilon ,l}) \xi_l \right) , \nu_{1,l} \right\rangle d\mu \right)^2 \\
& \leq & \left(  \left\Vert (1-\eta_{\epsilon ,l}) \partialbar\xi_l \right\Vert_{L^2}^2 +CR'^{-1} +\Psi (\epsilon | R') \right) \cdot \left\Vert \chi_{R'} \nu_{1,l} \right\Vert_{L^2}^2 \\
& \leq & \left( c^{-1} -\Psi(l^{-1}) \right) \left(  \left\Vert (1-\eta_{\epsilon ,l}) \partialbar\xi_l \right\Vert_{L^2}^2 +CR'^{-1} +\Psi (\epsilon | R') \right) \\
& & \cdot \left( \left\Vert \partialbar\left( \chi_{R'} \nu_{1,l} \right) \right\Vert_{L^2}^2 + \left\Vert \partialbar^* \left( \chi_{R'} \nu_{1,l} \right) \right\Vert_{L^2}^2 \right) \\
& \leq & \left( c^{-1} -\Psi(l^{-1}) \right) \left(  \left\Vert (1-\eta_{\epsilon ,l}) \partialbar\xi_l \right\Vert_{L^2}^2 +CR'^{-1} +\Psi (\epsilon | R') \right) \\
& & \cdot \left( \left\Vert \partialbar^*  \nu_{l} \right\Vert_{L^2}^2  + CR'^{-1}\right) .
\end{eqnarray*}
Note that $X_{sing} \cap B_{2R'} (p_\infty ) \subset \cap_{r\in (0,1) } \mathcal{S}^{n-2}_{\epsilon' ,r } (X) $ for some constant $\epsilon'>0 $. Letting $l\to\infty $. By choosing suitable constants $R'\to\infty $ and $\epsilon\to 0$ as $l\to\infty $, we get the formula (\ref{l2step2formula}). By using the Hahn-Banach and Riesz representation theorems, we can find a solution $u$ of this equation $\partialbar (u-\xi )=0$ in the $L^2$ closure of the vector space spanned by $\{ \star \partialbar \beta \}_{\beta} $ satisfying (\ref{l2propformula}). 

Finally, we have to show that (\ref{l2propformula}) holds for all smooth $\linebundle'_\infty $-valued $(0,1)$-form $\alpha\in L^2 $ when the non-collapsing condition $\Vol (B_1 (x) ) \geq v $ holds for all $x\in X$. The idea is to construct a sequence of smooth sections on $M_l $ to converge to a solution on $X$. By the partial $C^0 $-estimate, we can find a positive integer $D\in\mathbb{N} $ such that $\rho_{\linebundle^D_l ,1 } \geq b>0 $ for all $l$, where the Hermitian line bundles $( \linebundle_l ,h_l )$ are the polarizations of $\left( M_l ,\omega_l \right) $, and $b>0 $ is a constant independent of $l$. 

Let $s_{p_l} \in H_{L^2}^0 \left( M_l ,\linebundle_l^D \right) $ such that $\Vert s_{p_l} \Vert_{L^2} =1 $ and $\rho_{\linebundle^D_l ,1 } = \Vert s_{p_l} (p_l ) \Vert^2 $. Set $Y_l =\{ s_{p_l } =0 \} \subset M_l $. It is clear that $\Vert s_{p_l} \Vert^2 $ converge to a continuous function $\Vert s_{\infty } \Vert^2 $ on $X$, $Y_l $ converge to a subvariety $Y_\infty = \{ s_{\infty } =0 \} \subset X $, where $s_\infty \in H_{L^2}^0 \left( X ,\linebundle_\infty^D \right) $. For each $l\in\mathbb{N}$, define $\phi_l = \frac{\rho_{\linebundle^D_l ,1 }}{\Vert s_{p_l} \Vert^2 } $. Then $\phi_l $ is a plurisubharmonic function on $M_l \sq Y_l $ for each $l$, and $\phi_l $ tends to $\infty $ as it approaches $Y_l $. By Lemma \ref{peaksecdecaylmm}, we see that $\lim_{\dist (x, p_l) \to\infty } \phi_l (x) =\infty $. It follows that for each constant $\epsilon'' >0$, there are Stein manifolds $U'_{l,\epsilon'' } \subset M_l \sq Y_l $ for sufficiently large $l$, satisfying that $U'_{l,\epsilon'' } $ containing $ B_{\epsilon''^{-1}} (p_l ) \sq B_{\epsilon''} (Y_l ) $, and $U'_{l,\epsilon'' } \subset B_{\epsilon'''^{-1}} (p_l ) \sq B_{\epsilon'''} (Y_l ) $ for some constant $\epsilon''' =\epsilon''' (\epsilon'' ) >0 $.

By the cut-off functions $\eta_{i,l} $ in the second step, we can define smooth $\widetilde{\linebundle}_l $-valued $(0,1)$-forms $\alpha_l $ on $B_{R'} (p_l ) \sq \cup_{k} B_{10^{-1} \epsilon } (x'_{k,l}) $ satisfying that $\alpha_l \to\alpha $ as $l\to\infty $, and $\left\Vert \partialbar \alpha_l \right\Vert_{L^2} \leq \Psi (l^{-1} | R' ,\epsilon ) $.

Apply the H\"ormander's $L^2$ estimate to $ \left( \chi_{R'} (1-\eta_{\epsilon ,l} ) \alpha_l \right) $, we can find a sequence of smooth $\widetilde{\linebundle}_l $-valued $(0,1)$-forms $\alpha'_l $ on $U'_{l,\epsilon'' } $ such that $\left\Vert \alpha'_l - \chi_{R'} (1-\eta_{\epsilon ,l} ) \alpha_l \right\Vert_{L^2} \leq \Psi (l^{-1} | R' ,\epsilon ) +\Psi (\epsilon |R') +CR'^{-1} $, and $\partialbar \alpha'_l =0 $. By using the H\"ormander's $L^2$ estimate again, one can obtain smooth $\widetilde{\linebundle}_l $-valued functions $u_{R' ,\epsilon ,l} $ on $U'_{l,\epsilon'' } $ satisfying that $\partialbar u_{R' ,\epsilon ,l} = \alpha'_l $ and $\Vert u_{R' ,\epsilon ,l} \Vert_{L^2}^2 \leq \left( c^{-1} +\Psi (l^{-1} ) \right) \left\Vert \alpha'_l \right\Vert^2_{L^2} $. Letting $l\to\infty $. By choosing suitable constants $R'\to\infty $, $\epsilon\to 0$ and $\epsilon''\to 0$ as $l\to\infty $, we can assume that $u_{R' ,\epsilon ,l} $ has a subsequence converging to an $\widetilde{\linebundle}_l $-valued function $u\in L^2 $ on $X_{reg} \sq Y_\infty $ in the $L_{loc}^2$ sense. Note that by $F_{i,l}\to F_i $, we can consider the convergence locally on a single holomorphic chart. Then we have $\partialbar u = \alpha $ on $X_{reg} \sq Y_\infty $, and $\Vert u \Vert_{L^2}^2 \leq  c^{-1}  \left\Vert \alpha \right\Vert^2_{L^2} $. Hence $u$ is smooth on $X_{reg} \sq Y_\infty $. Clearly $\partialbar \alpha =0 $ implies that for each $x\in X$, we can find a smooth $\widetilde{\linebundle}_l $-valued function $u_x \in L^2 $ around $x$, satisfying $\partialbar u_x =\alpha $. Since $Y_\infty $ is a subvariety of $X$, $u\in L^2 $ and $\partialbar u=\alpha $ show that $u$ can be extended to a smooth $\widetilde{\linebundle}_l $-valued function on $X$, satisfying $\partialbar u =\alpha $ \cite[Proposition 1.14]{oh1}. It remains to prove that $\mu \left( Y_\infty \right) =0 $ now. It follows form $n! d\mu = \omega^n_x = \left( -\frac{\sqrt{-1}}{2\pi } \partial\partialbar \log h_\infty \right)^n $, and the proof is complete.
\end{proof}

In the above argument, we used the following lemma.

\begin{lmm}
\label{peaksecdecaylmm}
Given $b ,v >0$ and $n\in\mathbb{N}$, there exists a constant $\epsilon_0 >0 $ with the following property. Let $(M,\omega )$ be an $n$-dimensional complete K\"ahler manifold, $L$ be a positive line bundle on $M$ equipped with a hermitian metric $h$ whose curvature form is $2\pi \omega $, and $s_x $ is the peak section at $x\in M$ (See Lemma \ref{lmmlowercontinuitybergman}). Suppose that $\rho_{\omega ,1} (x)\geq b$, $\Ric(\omega ) \geq - \omega $ and $  \Vol\left( B_{1} (x) \right) \geq v $. Then
$$ \int_{M \sq B_{\epsilon^{-1}} (x) } \left\Vert s_x \right\Vert^2 dV_{\omega } \leq \Psi \left( \epsilon \big| n,b ,v \right) ,\;\; \forall \epsilon \in \left( 0, \epsilon_0 \right). $$
\end{lmm}

\begin{proof}
Choose a cut-off function $\psi :\mathbb{R} \to \mathbb{R} $ such that $\psi (t)=1 $ for $t\leq \frac{1}{2} $, $\psi (t)=0 $ for $t\geq 1 $, and $-5\leq -\psi' \leq 0 $. By smoothing $ \psi \left( \epsilon d_x \right) $, where $d_x$ is the distance function from $x$, we can find a function $\eta\in C^{\infty} (M)$, such that $\eta =1 $ on $B_{\frac{1}{2\epsilon} } (x) $, $\eta =0 $ on $M\sq B_{\epsilon^{-1}} (x) $, and $|\nabla \eta | \leq 10\epsilon$. It follows that $$\int_{M } \left\Vert \partialbar \left( \eta s_x \right) \right\Vert^2 dV_{\omega } \leq 100 \epsilon^2 \int_{M } \left\Vert s_x \right\Vert^2 dV_{\omega } \leq 100 \epsilon^2 .$$
By H\"ormander's $L^2$ estimate, there exists a smooth $L$-valued section $u$ such that $ \partialbar u =\partialbar \left( \eta s_x \right) $, and
\begin{eqnarray}
\label{uintegralsmall}
\int_{M } \left\Vert u \right\Vert^2 dV_{\omega }  \leq  \int_{M } \left\Vert \partialbar \left( \eta s_x \right) \right\Vert^2 dV_{\omega } \leq 100 \epsilon^2 .
\end{eqnarray}
Let $s=\eta s_x - u \in H^0 \left( M,L \right) $. Then Lemma \ref{sobolevlmm} shows that there exists a positive constant $C = C (n,\Lambda ,v)$ such that $\left\Vert s_x \right\Vert +\left\Vert Ds \right\Vert +\left\Vert Ds_x \right\Vert \leq C $ on $B_1 (x)$, and hence we can conclude that $\left\Vert Du \right\Vert = \left\Vert d\eta \otimes s_x + \eta Ds_x -Ds \right\Vert \leq C $ on $B_1 (x)$. Thus (\ref{uintegralsmall}) implies that $\left\Vert u(x) \right\Vert \leq \Psi \left( \epsilon \big| n,\Lambda ,v \right) $. By definition of $s_x$, we have
\begin{eqnarray*}
\int_{M} \left\Vert s \right\Vert^2 dV_{\omega } & \geq & \frac{\left\Vert s (x) \right\Vert^2}{\left\Vert s_x (x) \right\Vert^2 } \\
& \geq &  \left( \frac{ \left\Vert s_x (x) \right\Vert - \left\Vert u(x) \right\Vert }{\left\Vert s_x (x) \right\Vert } \right)^2 \\
& \geq &  1 - \Psi \left( \epsilon \big| n,b ,v \right).
\end{eqnarray*}
Since $s=\eta s_x - u $, it follows that
\begin{eqnarray*}
\left( \int_{M} \left\Vert s \right\Vert^2 dV_{\omega } \right)^{\frac{1}{2}} & \leq & \left( \int_{M } \left\Vert \eta s_x \right\Vert^2 dV_{\omega } \right)^{\frac{1}{2}} + \left( \int_{M } \left\Vert u \right\Vert^2 dV_{\omega } \right)^{\frac{1}{2}} \\
& \leq &  \left( 1- \int_{M \sq B_{\epsilon^{-1}} (x) } \left\Vert s_x \right\Vert^2 dV_{\omega } \right)^{\frac{1}{2}} + 10\epsilon .
\end{eqnarray*}
Consequently, $$1- \int_{M \sq B_{\epsilon^{-1}} (x) } \left\Vert s_x \right\Vert^2 dV_{\omega } \geq 1 - \Psi \left( \epsilon \big| n,b ,v \right) ,$$
and the lemma follows.
\end{proof}

We conclude this section by pointing out that if we relax the requirement for constants, then the line bundle $\linebundle'_\infty $ in Proposition \ref{propl2estimateghlimit} need not be a limit line bundle under the additional conditions of orthogonal bisectional curvature lower bound and non-collapsing everywhere. Let $ (X,d,p)  $ be the pointed Gromov-Hausdorff limit of a sequence of pointed complete K\"ahler manifolds $\left( M_l ,\omega_l ,p_l \right) $ with $\Ric \left( \omega_l \right) \geq -\Lambda \omega_l $ and $OB\geq -\Lambda $, $\forall l\in\mathbb{N} $, where $\Lambda >0$ is a constant. Assume that $\Vol \left( B_1 \left( x \right) \right) \geq v $, $ \forall x\in X $, where $ v>0$ is a constant.

\begin{coro}
\label{corol2thmlimit}
Let $\linebundle'_\infty $ be a line bundle on $X$ with a continuous Hermitian metric $h'_\infty $. Then there are constants $C_1 ,C_2 >0 $ with the following property. Suppose that $ \varphi \in L^1_{loc} (X) $ can be approximated by a decreasing sequence of smooth function $\left\lbrace \varphi_i \right\rbrace_{i=1}^{\infty} $ on $X$ satisfying the conditions in Proposition \ref{propl2estimateghlimit} with constant $C_1$. Then for any $\linebundle'_\infty $-valued smooth function $\xi \in L^2$ on $X$, there exists an $\linebundle'_\infty $-valued function $u\in L^{2}$ such that $\partialbar u= \partialbar \xi $ and 
\begin{equation}
\label{l2propformulacoro}
\int_{X} \Vert u \Vert^{2} e^{-\varphi } d\mu \leq C_2 \int_{X} \Vert \partialbar \xi \Vert^{2} e^{-\varphi } d\mu .
\end{equation}
\end{coro}

\begin{proof}
By \cite[Theorem 1.1]{leetam1}, there exists a subsequence of $\{ M_l \} $ such that the complex structures converging to a complex manifold structure on $X$. Then \cite[Lemma 3.4]{leetam1} shows that for each $R>0$, we can find a constant $T>0$ and a local K\"ahler-Ricci flow solution $\omega (t)$ on $B_{d;R+10} (p) \times (0,T] $ satisfying that $ B_{\omega (t);R+5} (p) $ converging to $B_{d;R+5} (p) $ in the Gromov-Hausdorff sense, $\Vol_{\omega (t)} (B_1 (x ,t) ) \geq C(n,v,\Lambda )^{-1} $ and $\Ric (\omega(t)) \geq -C(n,v,\Lambda ) \omega (t) $, $ \forall x\in B_{d;R } (p) $. Note that the line bundle $\linebundle'_\infty $ is also a holomorphic line bundle on $ (B_{d,R } (p) ,\omega (t) ) $. Hence we can follow the arguments in Proposition \ref{propl2estimateghlimit}.
\end{proof}

\section{Applications }
\label{applicationsection}

We will give some applications of the previous results in this section.

\subsection{Convergence of Fubini-Study currents }

Now we are in place to prove Proposition \ref{propfscurrentdistribution}.

\vspace{0.2cm}

\noindent \textbf{Proof of Proposition \ref{propfscurrentdistribution}: }
By the standard result of the Monge-Amp{\'e}re operator of locally bounded plurisubharmonic functions \cite[Corollary 3.6]{dm1}, we only need to show that $m^{-1} \log \left( \rho_{\linebundle'_\infty ,m } \right) $ converge to $0$ locally uniformly.

Let $x\in X$. Choosing a sequence of points $M_l \ni x_l \to x $ and a sequence of holomorphic maps $F_{l} : B_{10r} (x_l ) \to \mathbb{C}^N $ converging to an injective holomorphic map $F : B_{10r} (x) \to\mathbb{C}^N $. Here $r>0$ is a constant. Set $\eta \in C^\infty_c \left( B_{10r} (x) \right) $ be a cut-off function such that $\eta =1$ on $B_{9 r} (x)$. For each $y\in B_{5 r} (x) $, we define $\alpha_y =\partialbar \left( \eta s_{m,y} \right) $ and $\varphi_{y} = \eta \left| F -F(y) \right|^{10n} $, where $s_{m,y} $ is the peak section at $y$. By Proposition \ref{propl2estimateghlimit}, we have 
$$\rho_{X,\linebundle'_\infty , m} \geq \left( 1-\Psi \left( m^{-1} |r,x \right) \right) \rho_{B_{10r} (x) ,\linebundle'_\infty ,m} ,$$ 
on $B_{5 r } (x)$, where $\rho_{U,\linebundle'_\infty ,m} $ is the $m$-th Bergman kernel of $(U ,\mu ,\linebundle'_\infty |_{U} ,h'_\infty |_{U} )$.

Without loss of generality, we can assume that $\linebundle'_\infty |_{B_{10 } (x)} $ is a trivial bundle, there exists a section $e_{x} \in H^{0}_{L^2} (B_{10 r } (x) ,\linebundle'_\infty ) $ satisfying $\left| \Vert e_x \Vert -1 \right|\leq \Psi (r|x) $. Note that $h'_\infty$ is continuous. By shrinking the value of $r$, one can obtain a K\"ahler potential $\psi_x $ on $B_{10 r } (x) $ such that $\left| \psi_x \right|\leq \Psi (r|x) $. Combining the partial $C^0 $ estimate with Lemma \ref{lmmlowercontinuitybergman} we conclude that $\rho_{B_{10r} (x) ,\mathbb{C} ,\psi_x ,m} \geq bm^{n} $ on $B_{5 r } (x)$ for sufficiently large $m$, where $b>0$ is a constant, and $\rho_{B_{10r} (x) ,\mathbb{C} ,\psi_x ,m} $ is the $m$-th Bergman kernel of $(U ,\mu ,\mathbb{C} ,e^{-2\pi \psi_x } )$. Clearly, $ \left| \log \Vert e_x \Vert \right| +\left| \psi_x \right| \leq \Psi (r|x) $. By the definition of Bergman kernels, 
$$ \rho_{B_{10r} (x) ,\linebundle'_\infty ,m} \geq e^{-m\Psi (r|x) } \rho_{B_{10r} (x) ,\mathbb{C} ,\psi_x ,m} \geq b e^{-m\Psi (r|x)} m^n ,$$
on $B_{5 r } (x)$, for sufficiently large $m$. It follows that 
$$ \liminf_{m\to\infty } m^{-1} \log \left( \rho_{\linebundle'_\infty ,m } \right) \geq -\Psi (r|x) ,$$
on $B_{5r} (x) $. By the definition, $\rho_{X,\linebundle'_\infty , m} \leq \rho_{B_{10r} (x) ,\linebundle'_\infty ,m} $, and hence we have 
$$ \limsup_{m\to\infty } m^{-1} \log \left( \rho_{\linebundle'_\infty ,m } \right) \leq \Psi (r|x) ,$$ 
on $B_{5r} (x) $ by a similar argument. It follows that for each $x\in X$ and any constant $\epsilon >0 $, there are constants $r_{x,\epsilon } >0$ and $m_{x,\epsilon } $, such that $ \left| m^{-1} \log \left( \rho_{\linebundle'_\infty ,m } \right) \right| \leq \epsilon $ on $B_{r_{x,\epsilon } } (x) $, $ \forall m\geq m_{x,\epsilon } $. Then we see that $m^{-1} \log \left( \rho_{\linebundle'_\infty ,m } \right) $ converge to $0$ locally uniformly.
\qed

\vspace{0.2cm}

\subsection{\texorpdfstring{$L^p$}{Lg} asymptotic expansion }
\label{lpestsection}
Now we prove Theorem \ref{lpestprop}.

\vspace{0.2cm}

\noindent \textbf{Proof of Theorem \ref{lpestprop}: }
We consider the case $p=1$ at first. Clearly, the unique polarization of the Euclidean space $\mathbb{C}^n$ is the Hermitian line bundle $\left( \mathbb{C} ,e^{-\pi|z|^2} \right) $, and the Bergman kernel $\rho_{\mathbb{C}^n ,1} =1 $. Theorem \ref{thmcontinuousbergmankernel} now shows that $$ \left| \rho_{r^{-2}\omega ,1} (x) -1 \right| \leq \Psi (\epsilon |n,\Lambda ) ,$$
for each $x\notin \mathcal{S}_{\epsilon ,  \epsilon^{-1} r}$, where $r>0$ is a constant. By Theorem \ref{jcwsjan1thm17coro}, we conclude that there exists a constant $C_1 =C_1 (n,\epsilon ,\Lambda ,v)$ such that
$$ \Vol \left( B_{\epsilon^{-1} r} \left( \mathcal{S}^{n-1}_{\epsilon ,\epsilon^{-1} r} (X) \right) \cap B_1 (x)  \right) \leq C_1 r^2 .$$
Let $r={m}^{-\frac{1}{2}} $. Then $\rho_{m\omega ,1} =m^{-n} \rho_{\omega ,m} $ implies that
\begin{eqnarray}
\Vol \left( \left\lbrace y\in  B_1 (x) : \left| m^{-n} \rho_{\omega ,m} (y) -1 \right| \geq \delta \right\rbrace \right) \leq C_2 {m}^{-1} , \label{alexrholowerest}
\end{eqnarray}
where $C_2 $ is a constant that only depends on $n,\delta ,\Lambda ,v$. To shorten notation, we denote by $\mathcal{B}_{m,\delta }$ the set $\left\lbrace y\in  B_1 (x) : \left| m^{-n} \rho_{\omega ,m} (y) -1 \right| \leq \delta \right\rbrace $ from now. Combining Lemma \ref{sobolevlmm} with (\ref{alexrholowerest}), we obtain
\begin{eqnarray*}
& & \int_{B_1 (x)} \left| m^{-n} \rho_{\omega ,m} - 1 \right|^{p} \omega^n \\
& \leq & \int_{B_1 (x) \sq \mathcal{B}_{m,\delta } } \left| m^{-n} \rho_{\omega ,m} - 1 \right|^{p} \omega^n + \int_{\mathcal{B}_{m,\delta } } \left| m^{-n} \rho_{\omega ,m} - 1 \right|^{p} \omega^n  \\
& \leq & \Vol \left( B_1 (x) \right) \delta^p + C^p_3 {m}^{-1}  ,
\end{eqnarray*}
where $C_3 $ is a constant that only depends on $n,\delta ,\Lambda ,v$. It follows that
$$ \left\Vert m^{-n} \rho_{\omega ,m} - 1 \right\Vert_{L^p \left( B_1 (x) \right) } \leq \Psi \left( \delta , \frac{1}{m} \bigg| n ,\Lambda ,v \right) ,$$
and this is precisely the assertion. \qed

\vspace{0.2cm}

\begin{rmk}
If we further assume that $\left| \Ric \right| \leq \Lambda $, then there are positive constants $m_0 = m_0 \left( n,p,\Lambda ,v \right)$, $C = C \left( n,p,\Lambda ,v \right)$, such that
\begin{eqnarray*}
\left\Vert m^{-n} \rho_{\omega ,m} - 1 \right\Vert_{L^p \left( B_1 (x) \right) }  & \leq & C \max\left\lbrace m^{-1}  , m^{-\frac{2}{p}} , m^{-1} \sqrt{ \log (m)} \delta_{p,2} \right\rbrace ,\;\; \forall m\geq m_0 .
\end{eqnarray*}
When $\left| \Ric \right| \leq \Lambda $, we can prove this stronger estimate by combining Anderson's $C^{1,\alpha }$ coordinates (\cite{anderson1}, Theorem 3.2), Jiang-Naber's Minkowski type estimate (\cite{jwsnb1}, Theorem 1.15) and Tian's peak section method (\cite{tg1}, Lemma 1.2). 
\end{rmk}

Now we discuss an example showing that there exists a sequence of $1$-dimensional polarized K\"ahler manifolds $\left( M,\omega_k ,L \right) $ with bounded sectional curvature and non-collapsing volume lower bound, such that for any sequence of functions $a_{1,k} \in L^1 \left( M,\omega_k \right) $, the uniform convergence condition
$$ \lim_{m\to\infty} \sup_{k\in\mathbb{N}} \left\Vert m^{1-n} \rho_{\omega_k ,m}  -m -a_{1,k } \right\Vert_{L^1} =0 $$
does not hold. It follows that there is no uniform second-order asymptotic expansion similar to above for the Bergman kernels.

\begin{exap}
\label{oscillation}
\upshape Let $M=\mathbb{C}P^1$, $L=\mathcal{O} (1)$, then we have two open sets 
\begin{eqnarray*}
U_0 & = & \left\lbrace \left[ 1 , w \right]\in \mathbb{C}P^1  \right\rbrace ,\\
U_1 & = & \left\lbrace \left[ z , 1 \right]\in \mathbb{C}P^1 : \left| z \right| <1 \right\rbrace
\end{eqnarray*} 
in $M$, such that $U_0 \cup U_1 =M $. Choose a radial cut-off function $\eta \in C_0^{\infty} \left( B_1(0) \right) \subset C_0^{\infty } \left( \mathbb{C} \right)$, s.t. $0\leq\eta\leq 1$ and $\eta =1$ on $B_{\frac{1}{2}} (0)$. For each $k\in\mathbb{N}$, we define $$\varphi_k = k^{-4} \sin (kz+k\zbar  )\sin (\sqrt{-1}kz-\sqrt{-1}k\zbar ) \eta (z) $$ on $U_1 $. Then $h_k = e^{\varphi_k } h_{0} $ gives a Hermitian metric on $L$, where $h_0 $ be the normal metric on $L$, i.e. $h_0 = \frac{1}{|w|^2 +1} $ on $U_0 $, and $h_0 = \frac{1}{|z|^2 +1}$ on $U_1$. Clearly, $\Ric(h_0 ) = 2\pi \omega_{FS} $ on $M$, and hence 
\begin{eqnarray*}
\Ric( h_k ) & = & 2\pi \omega_{FS} - \sqrt{-1}  \partial\partialbar \varphi_k .
\end{eqnarray*}
For sufficiently large $k$, $\Ric\left( h_k  \right) $ is also a K\"ahler form. Let $g_k $ be the K\"ahler metric corresponds to $\Ric\left( h_k  \right) $. Recall that $$R_{1\bar{1} 1\bar{1}} = -\frac{\partial^2 g_{1\bar{1}}}{\partial z \partial \zbar } - g^{1\bar{1}} \frac{\partial g_{1\bar{1}}}{\partial z } \frac{\partial g_{1\bar{1}}}{\partial \zbar } ,$$
we can find a constant $C>0 $ such that on $U_1 $, 
\begin{eqnarray*}
\left| R_{k, 1\bar{1} 1\bar{1}} + \frac{64 k^4}{\pi }   \varphi_k  - R_{FS,1\bar{1} 1\bar{1}} \right| \leq \frac{C}{k} , \label{rkfs}
\end{eqnarray*}
where $ R_{FS,1\bar{1} 1\bar{1}} $ is the curvature of Fubini-Study metric on $U_1 $, and $ R_{k, 1\bar{1} 1\bar{1}} $ is the curvature of $g_k$ on $U_1 $. It is easy to check that there is a constant $\delta >0$ satisfies that 
\begin{eqnarray}
\liminf_{k\to\infty } \left\Vert  k^4 \varphi_k -f  \right\Vert_{L^1 (M)} >\delta ,\label{oscillationvarphi}
\end{eqnarray}
for each given $f\in L^1 (M)$.

If there is a uniform sequence $b_m $ such that $\lim_{m\to\infty} b_m =0$, and 
\begin{eqnarray*}
\left\Vert  m^{1-n}\rho_{\omega_k ,m} -m - a_{1,k} \right\Vert_{L^1 } \leq b_m  ,
\end{eqnarray*}
for any sufficiently large $k$ and $m$, then by Lu's work \cite{zql1}, we have $a_{1,k} = \frac{S\left( \omega_k \right)}{2}$, where $S\left( \omega_k \right) $ is the scalar curvature of $\left( M, \omega_k \right)$. Since $g_k $ converges to $g$ in $C^\infty $-topology as $k\to\infty $, $m^{1-n}\rho_{\omega_k ,m} -m$ is also convergence in $C^\infty $-topology as $k\to\infty $, for each given $m$. It follows that the sequence of scalar curvatures of $g_k$ is convergence in $L^1$ norm as $k\to\infty $, contrary to \upshape{(\ref{oscillationvarphi})}.
\end{exap}

\subsection{\texorpdfstring{$C^0$}{Lg} asymptotic estimate}
\label{sectionc0asymptoticestimate}

This part is devoted to the proof of Proposition \ref{refinementofbkk}.

Before proving Proposition \ref{refinementofbkk}, we need to consider the complex structure of the tangent cones of the limit spaces. This generalizes a result proved by Lott under the lower bound of bisectional curvature \cite{lott1}. Note that the following limit space is a complex manifold \cite[Theorem 1.1]{leetam1}. 

\begin{lmm}
\label{tangentconeOBCn}
Let $\left( X ,d ,p \right) $ be the Gromov-Hausdorff limit of a sequence of pointed complete polarized K\"ahler manifolds $\left( M_l ,\omega_l ,p_l \right) $ with $\Ric \left( \omega_l \right) \geq -\Lambda \omega_l $, $OB\geq -\Lambda $ and $\Vol \left( B_1 \left( p_l \right) \right) \geq v $, $\forall l\in\mathbb{N} $, where $\Lambda ,v>0$ are constants. Let $T_p X$ be a tangent cone of $X$ at $p$. Then $T_p X$ is biholomorphic to $\mathbb{C}^n $, and $\frac{1}{2} r^2 $ is a K\"ahler potential on $T_p X$, where $r $ is the distance from $p$.
\end{lmm}

\begin{proof}
From \cite[Theorem 1.1]{leetam1}, we can assume that the complex structures of $M_l $ converging to a complex manifold structure on $X$, and hence $T_p X$ is also a complex manifold. Note that there are positive constants $\mu_l $ such that $\mu_l \to 0 $, and $\left( M_l ,\mu^{-2}_l \omega_l ,p_l \right) $ converging to $(T_p X  ,p ) $ in the pointed Gromov-Hausdorff sense. 

By \cite[Theorem 1.1]{lggs2}, we can find holomorphic functions $\{ f_j \}_{j=1}^{N} $ on $T_p X$ satisfying that $r\partial_r f_j =\lambda_j f_j $ for some constant $\lambda_j >0 $, and $F_N =(f_1 , \cdots ,f_N )$ gives a proper holomorphic embedding $F_N : T_p X \to \mathbb{C}^N $. Without loss of generality, we can assume that the map $F= (f_1 , \cdots ,f_n ) :T_p X\to \mathbb{C}^n $ is regular at $p\in T_p X $. Then the homogeneity of $f_j $ implies that $F$ is proper and degree $1$. It follows that $F$ is also an embedding, and hence $F$ is biholomorphic.

By the standard Cheeger-Colding theory, for each $R>0$, we can find a sequence of smooth functions $\phi_l $ on $B_R (p_l ) \subset \left( M_l ,\mu^{-2}_l \omega_l ,p_l \right) $ such that $ \left| \phi_l -\frac{1}{2} d_l^2 \right| \leq \Psi (l^{-1} ) $ on $B_R (p_l ) $, and 
$$ \int_{B_R (p_l )} \left| \sqrt{-1} \partial\partialbar \phi_l -\mu^{-2}_l \omega_l \right|_{\mu^{-2}_l \omega_l }^2 \mu^{-2n}_l \omega^n_l \leq \Psi (l^{-1} ) .$$
Then we see that $\frac{1}{2} r^2 $ is a K\"ahler potential on $B_R (p ) \subset T_p X$, $\forall R>0 $. This completes the proof.
\end{proof}

\vspace{0.2cm}

\noindent \textbf{Proof of Proposition \ref{refinementofbkk}: }
We argue by contradiction. Suppose that there exist a sequence of pointed polarized K\"ahler manifolds $\left( M_l ,\omega_l ,\linebundle_l ,h_l ,p_l \right) $ with $BK \geq -K $ and $\Vol \left( B_1 \left( p_l \right) \right) \geq v $, $\forall l\in\mathbb{N} $, where $K ,v>0$ are constants, and a strictly increasing sequence of integers $\left\lbrace j_l \right\rbrace_{l=1}^{\infty}$ such that $ \rho_{M_l ,j_l \omega_l ,\linebundle^{j_l}_l ,h^{j_l}_l ,1} \left( p_l \right) = j^{-n}_{l} \rho_{M_l ,\omega_l ,\linebundle_l ,h_l ,j_l} \left( p_l \right) \to 0 $ as $l\to\infty $. By choosing a subsequence, we can assume that $ \left( M_l ,j_l \omega_l ,p_l \right) $ converge to a pointed metric space $(X,d,p)$ with a normal complex space structure, and $ ( \linebundle^{j_l}_l ,h^{j_l}_l ) $ converge to a line bundle $\linebundle_\infty $ with a continuous Hermitian metric $h_\infty $. Actually, $X$ is a complex manifold \cite[Theorem 1.1]{leetam1}. By Theorem \ref{thmcontinuousbergmankernel}, we can conclude that $ \rho_{X ,\mu ,\linebundle_\infty ,h_\infty ,1} \left( p \right) = 0 $, where $\mu $ is the $2n$-dimensional Hausdorff measure on $X$.

Now we consider the sequence of pointed polarized K\"ahler manifolds $\left( M_l , j_{l,k} \omega_l ,\linebundle_l^{j_{l,k}} ,h_l^{j_{l,k}} ,p_l \right) $ for each $k\in\mathbb{N} $, where $j_{l,k} = \lfloor k^{-1} j_l \rfloor $, and $\lfloor \cdot \rfloor $ is the greatest integer function. It is clear that $\left( M_l , j_{l,k} \omega_l ,\linebundle_l^{j_{l,k}} ,h_l^{j_{l,k}} ,p_l \right) $ converge to $ \left( X,k^{-\frac{1}{2}} d ,\linebundle_{\infty ,k} ,h_{\infty ,k} ,p  \right) $, where $\left( \linebundle^k_{\infty ,k} ,h^k_{\infty ,k} \right) = \left( \linebundle_{\infty } ,h_{\infty } \right) $. Then Lemma \ref{sobolevlmm} implies that $ \rho_{X ,k^{-n} \mu ,\linebundle_{\infty ,k} ,h_{\infty ,k} ,1} \left( p \right) = 0 $, $\forall k\in\mathbb{N} $. But $ \left( X,k^{-\frac{1}{2}} d , p  \right) $ converge to a metric cone $(V,d' ,o)$ as $k\to\infty$, and hence their complex structures converge to $\mathbb{C}^n $ with K\"ahler potential $\frac{r^2}{2}$, where $r$ is the distance function from $o$. Clearly, $ (\linebundle_{\infty ,k} , h_{\infty ,k} ) $ converge to the trivial bundle with Hermitian metric $e^{-\pi r^2}$. Apply Theorem \ref{thmcontinuousbergmankernel} again, we have $ \rho_{V,\mathbb{C} ,e^{-\pi r^2} ,1} \left( o \right) = 0 $. Then $e^{-\pi r^2} \in L^2 (V,\mu_V ) $ shows that the frame $e\in H^0_{L^2} \left( V,\mathbb{C} \right) $, where $\mu_V$ is the $2n$-dimensional Hausdorff measure on $(V,d')$. Note that there exists a constant $c>0$ such that $\Vol (B_r (o)) =cr^{2n} $, $\forall r>0$.  It follows that $ \rho_{V,\mathbb{C} ,e^{-\pi r^2} ,1} \left( o \right) \geq \Vert e^{-\pi r^2} \Vert_{L^2}^{-2} >0 $, contradiction.
\qed

\vspace{0.2cm}

\subsection{Lower bound of partial \texorpdfstring{$C^0$}{Lg} estimate}
\label{Calculationsonorbifolds}

In this section, we consider the upper bound of the lower bound of Bergman kernel on orbifold with cyclic quotient singularity. We start with the calculation of the Bergman kernel on the model space $ \mathbb{C}^n /\mathbb{Z}_q $. The following result is known (see \cite{dailiuma1}, \cite{mamari1}), but we include a proof.

\begin{lmm}
\label{lmmcircleorbifoldbgmkn}
Let $\left( X ,0\right) = \mathbb{C}^n /\mathbb{Z}_q $ be a flat orbifold with cyclic quotient singularity, and $L$ be a trivial line bundle on $X $ equipped with a hermitian metric $h$ whose curvature form is $2\pi \omega $, where $\mathbb{Z}_q$ is generated by the action
\begin{displaymath}
\sigma = \left\{ \begin{aligned}
\mathbb{C}^n & \longrightarrow & \mathbb{C}^n \quad\quad\quad\quad\quad\quad\quad\quad\quad ,\\
\left( z_1 ,\cdots ,z_n \right) & \longmapsto & \left( e^{\frac{2p_1 \pi}{q_1} \sqrt{-1} } z_1 ,\cdots , e^{\frac{2p_n \pi}{q_n} \sqrt{-1} } z_n \right) ,
\end{aligned} \right.
\end{displaymath}
$p_i$, $q_i$ are co-prime for $i=1,\cdots ,n$, and $q$ is the least common multiple of $ q_1 ,\cdots ,q_n  $. Then we have
\begin{eqnarray}
\rho_{\omega_{Euc} ,1} (p(z)) =  e^{-\pi |z|^2}  \sum_{j=0}^{q-1} e^{ \pi \sum_{i=1}^{n} |z_i |^2 e^{\frac{ 2 j p_i \pi\sqrt{-1} }{q_i}} } , \label{equacircletionbgmkn}
\end{eqnarray}
where $p$ is the quotient map $\mathbb{C}^n \to \mathbb{C}^n /G $, and $\omega_{Euc}$ is the Euclidean metric on $\mathbb{C}^n$.
\end{lmm}

\begin{proof}
Since $L$ is trivial, we can find a non-vanishing section $e_L \in H^0 \left( X,L \right) $. Then we consider the map between global sections induced by pullback:
$$ \phi_p : H^0 \left( X,L \right) \longrightarrow H^0 \left( \mathbb{C}^n ,p^* L \right)  .$$
It is easy to show that $ \phi_p $ is an embedding, and the image of $\phi_p $ is
$$ \phi_p \left( H^0 \left( X,L \right) \right) = \left\lbrace f \phi_p \left( e_L \right) : f\in  \mathcal{O} \left( \mathbb{C}^n \right)^G \right\rbrace  ,$$
where $ \mathcal{O} \left( \mathbb{C}^n \right)^G $ is the class of $G$-invariant holomorphic functions. 

Write $a(z) = h\left( e_L (p(z)) , e_L (p(z)) \right) $, then we have $ -\partial\partialbar \log (a) = 2\pi p^* \omega = \pi \partial\partialbar |z|^2 $, and $a\circ\sigma  =a$, $\forall \sigma \in G$. It follows that $\psi = \log (a) +\pi |z|^2 $ is pluriharmonic, and hence we can find a holomorphic function $f$ such that $Re(f)=\psi $. Then we have $ Re(f-f\circ\sigma ) =0 $, $\forall \sigma \in G$. It shows that $f-f\circ\sigma$ is a constant, and $f(0)=f(\sigma (0))$ implies that $f\in\mathcal{O} \left(\mathbb{C}^n \right)^G $. Replacing $e_L $ by $e^{-\frac{f}{2}} e_L$, we can assume that $a=e^{-\pi |z|^2}$. It is easy to check that
\begin{eqnarray*}
  \left\lbrace c_{j_1 ,\cdots ,j_n } z_1^{j_1} \cdots z_n^{j_n} e_L \right\rbrace_{\begin{subarray}{l}
j_1 ,\cdots ,j_n \in \mathbb{Z}_+  \\
\frac{j_1 p_1 }{q_1 } + \cdots + \frac{j_n p_n }{q_n } \in \mathbb{Z}
\end{subarray}}
\end{eqnarray*}
is an $L^2 $ orthonormal basis in $A^2 \left( X ,L \right) $, where
\begin{eqnarray*}
 c^{-2}_{j_1 ,\cdots ,j_n } = |G|^{-1} \int_{\mathbb{C}^n } e^{-\pi |z|^2} \left| z_1^{j_1} \cdots z_n^{j_n} \right|^2 dV_{\mathbb{C}^n},
\end{eqnarray*}
and $dV_{\mathbb{C}^n} = \left( \frac{ \sqrt{-1} }{2} \right)^n dz_1 \wedge d\zbar_1 \wedge\cdots\wedge dz_n \wedge d\zbar_n$ is the volume form on $\mathbb{C}^n$. Hence we have
\begin{eqnarray}
\rho_{\omega_{Euc} ,1} (p(z)) & = &  q e^{-\pi |z|^2} \sum_{\begin{subarray}{l}
j_1 ,\cdots ,j_n \in \mathbb{Z}_+  \\
\frac{j_1 p_1 }{q_1 } + \cdots + \frac{j_n p_n }{q_n } \in \mathbb{Z}
\end{subarray}}
\frac{ \pi^{j_1 +\cdots +j_n } \left| z_1^{j_1} \cdots z_n^{j_n} \right|^2 }{ \prod_{l=1}^{n} j_1 ! } . \label{fkpin}
\end{eqnarray}

Comparison of (\ref{fkpin}) and the Taylor expansion of exponential function shows that 
\begin{eqnarray*}
\rho_{\omega_{Euc} ,1} (p(z)) & = & e^{-\pi |z|^2} \sum_{k=0}^{q-1} \sum_{\begin{subarray}{l}
j_1 ,\cdots ,j_n \in \mathbb{Z}_+  \\
\frac{j_1 p_1 }{q_1 } + \cdots + \frac{j_n p_n }{q_n } \in \mathbb{Z}
\end{subarray}} \prod_{l=1}^{n} \frac{  \pi \left|  z_l \right|^{2 j_l} e^{\frac{ 2 j_l p_l k \pi\sqrt{-1} }{q_l}} }{ j_1 ! } \\
& = & e^{-\pi |z|^2}  \sum_{k=0}^{q-1} e^{ \pi \sum_{l=1}^{n} |z_l |^2 e^{\frac{ 2 k p_l \pi\sqrt{-1} }{q_l}} } .
\end{eqnarray*}
\end{proof}

We give some elementary estimates in the following result to estimate $\rho$ on a fixed-line.

\begin{lmm}
\label{numberlmm1}
Under the assumptions of Lemma \ref{lmmcircleorbifoldbgmkn}, if we assume in addition that $q\geq 3$, then there are constants $r_1 ,\cdots ,r_n \geq 0$, $1\leq j\leq q-1$ such that
\begin{eqnarray}
\sum_{l=1}^{n} r_l \cos \left( \frac{2jp_l \pi }{q_l} \right) > \sum_{l=1}^{n} r_l \cos \left( \frac{2k p_l \pi }{q_l} \right) ,\;  \label{rcosestimate}
\end{eqnarray}
for each $k\in \left\lbrace 1,\cdots ,q-1 \right\rbrace \sq \left\lbrace j,q-j \right\rbrace $, and
\begin{eqnarray}
\sum_{l=1}^{n} r_l \sin \left( \frac{2jp_l \pi }{q_l} \right) \neq 0 . \label{rsinestimate}
\end{eqnarray}
\end{lmm}

\begin{proof}
When $q_l = q$ for some $l$, we can assume that $q_1 = q$. Let $r_1 =1 $, $r_l =0 $ for each $l>1$, and let $j$ be the unique positive integer such that $q | jp_1 -1$ and $j<q$. Then $q>3$ implies the estimates (\ref{rcosestimate}) and (\ref{rsinestimate}).

In general, we would like to use induction on dimension. For $n=1$, we have $q_1 =q$, and we have proved the lemma in this case. Assume that the statement holds for $n=m$ and that we have $n=m+1$. It follows from elementary number theory that there exists an integer $ t\geq 0 $, such that $2^t \mid q_l $ for some $l$, but $2^{t+1} \nmid q_l $ for each $l$. We can assume that $2^t \mid q_{1} $. Let $q'$ be the least common multiple of $q_1 ,\cdots ,q_{m} $. Then we can conclude that $\frac{q}{q'} \neq 2 $.

By the induction hypothesis, there are constants $r_1 ,\cdots ,r_m \geq 0$, $1\leq j' \leq q' -1$ such that
$$\sum_{l=1}^{m} r_l \cos \left( \frac{2j' p_l \pi }{q_l} \right) > \sum_{l=1}^{m} r_l \cos \left( \frac{2k p_l \pi }{q_l} \right) ,\;$$
for each $k\in \left\lbrace 1,\cdots ,q'-1 \right\rbrace \sq \left\lbrace j',q' -j' \right\rbrace $, and
$$\sum_{l=1}^{m} r_l \sin \left( \frac{2j' p_l \pi }{q_l} \right) \neq 0 .$$
If $\frac{q}{q'} =1 $, the proof is finished. Now we assume that $\frac{q}{q'} \geq 3 $. A trivial verification shows that there are integers $p'_{m+1}$ and $q'_{m+1}$, such that $ \frac{p'_{m+1}}{q'_{m+1}} = \frac{p_{m+1} q' }{q_{m+1}} $, $q'_{m+1} =\frac{q}{q'} \geq 3 $, and $p'_{m+1} ,q'_{m+1}$ are co-prime. We now apply the induction hypothesis again to conclude that there exists a constant $1\leq j'' \leq q'_{m+1} -1$ such that
$$ \cos \left( \frac{2j'' p'_{m+1} \pi }{q'_{m+1}} \right) >  \cos \left( \frac{2k p'_{m+1} \pi }{q'_{m+1}} \right) ,\;$$
for each $k\in \left\lbrace 1,\cdots ,q'_{m+1}-1 \right\rbrace \big\sq \left\lbrace j'',q'_{m+1} -j'' \right\rbrace $, and
$$ \sin \left( \frac{2j'' p'_{m+1} \pi }{q'_{m+1}} \right) \neq 0 .$$
Let $j= j'' q' +j' $, $r_{m+1} = 1+ \left| \frac{\sum_{l=1}^{m} r_l \sin \left( \frac{2j' p_l \pi }{q_l} \right)}{\sin \left( \frac{2j'' p'_{m+1} \pi }{q'_{m+1}} \right)} \right| $. It is easy to check that (\ref{rcosestimate}) and (\ref{rsinestimate}) hold in this case. This concludes the induction step, and the proof is complete.
\end{proof}

As a corollary, we have:

\begin{lmm}
\label{lmmxy1}
Under the assumptions of Lemma \ref{lmmcircleorbifoldbgmkn}, if we assume in addition that $q\geq 3$, then we can find $x\in X $ such that $\rho_{\omega_{Euc} ,1} (x) < 1 $.
\end{lmm}

\begin{proof}
Let $ z=\left( t \sqrt{r_1 } ,\cdots , t \sqrt{r_n} \right) $ for some $t>0$. Then (\ref{equacircletionbgmkn}) implies that
\begin{eqnarray*}
\rho_{\omega_{Euc} ,1} \left( p\left( z \right) \right) & = & 1 + 2 e^{\pi t^2  \sum_{l=1}^{n} r_l \left( \cos \left( \frac{2jp_l \pi }{q_1} \right) -1 \right) } \cos \left( t^2 \sum_{l=1}^{n} r_l \sin \left( \frac{2jp_l \pi }{q_1} \right) \right) \\
& + & o\left( e^{\pi t^2  \sum_{l=1}^{n} r_l \left( \cos \left( \frac{2jp_l \pi }{q_1} \right) -1 \right) } \right)
\end{eqnarray*}
as $t\to\infty $. Hence we can find a sufficiently large integer $k$, such that $\rho_{\omega_{Euc} ,1} \left( p\left( z \right) \right) <1 $ for $$t  = \sqrt{ \frac{ (2k+1 )\pi }{ \sum_{l=1}^{n} r_l \sin \left( \frac{2jp_l \pi }{q_1} \right) } } .$$ Then the result follows.
\end{proof}

By Ding-Tian's argument in \cite{dt1}, there exists a sequence of K\"ahler–Einstein Del Pezzo surfaces converges to a K\"ahler–Einstein orbifold $\mathbb{C}P^2 /\mathbb{Z}_3 $ with three $A_2$ singularities in the Gromov-Hausdorff sense. See also \cite{oss1}. Then we can prove Theorem \ref{dscounterexp} by approximating these cones. By Theorem \ref{thmcontinuousbergmankernel} and Corollary \ref{corouniquehermitianmetric}, the proof of Theorem \ref{refinementofbkk} boils down to the following proposition.

\begin{prop}
\label{propndimobfdcounterexp}
Let $(X,\omega )$ be an $n$-dimensional K\"ahler orbifold with cyclic quotient singularities, and $\linebundle $ be an ample line bundle on $X$ equipped with a hermitian metric $h$ whose curvature form is $2\pi \omega_X $. Assume that $\left| G_x \right| \geq 3 $ for some $x\in X$, then we can find a constant $\epsilon >0$ and a sequence $x_j $ in $X$, such that
$$ \limsup_{j\to\infty } \left( j^{-n}\rho_{\omega_X ,j } \left( x_j \right) \right) \leq 1 -\epsilon .$$
\end{prop}

\begin{proof}
We argue by contradiction. Suppose that there exists a strictly increasing sequence of integers $\left\lbrace j_k \right\rbrace_{k=1}^{\infty}$ such that $  j^{-n}_{k} \inf_{x\in X} \rho_{\omega_X ,j_k} (x) \to 1$ as $k\to\infty $. Choose a point $x\in X$ such that $\left| G_x \right|\geq 3$. Then there exists a neighborhood $U_x$ of $x$ such that $U_x \cong B_r (0) /G $, where $G$ is a finite subgroup of $U(n)$, $r>0$ is a constant, and $G\cong G_x $. Clearly, for each $m\in\mathbb{N}$, we have $\rho_{\omega_X ,m} (x) =0 $ if $L^{m} \big|_{U_x}$ is not trivial. By the hypothesis, we can assume that $L^{j_k} \big|_{U_x}$ is trivial line bundle, $\forall k\in \mathbb{N} $.

Let $f: B_r (0) /G \to U_x $ be an isomorphism, and $p : \mathbb{C}^n \to \mathbb{C}^n /G $ be the quotient map. Then the form $ \omega = p^* \circ f^* \left( \omega_X \right) $ can be extended to be a $G$-invariant K\"ahler metric on $B_r (0)$, and we can assume that $\omega = \omega_{Euc} + O(|z|) $, where $|z|\to 0$. Then we see that $\left( X,j_k \omega_X ,\linebundle^{j_k} ,x \right) $ converge to $\left( \mathbb{C}^n /G , \omega_{Euc} ,\mathbb{C} ,0 \right) $. By the assumption, $\rho_{\omega_{Euc} ,1} \geq 1$ on $\mathbb{C}^n /G $, contradiction.
\end{proof}

In the case $n=1$, $\left|\Ric \right|\leq \Lambda $ implies that the sectional curvature is bounded, and hence the $C^0$ asymptotic expansion $\rho_{\omega ,m} \sim m^n $ is uniform \cite{sxz1}. The following result is an analogy of Theorem \ref{dscounterexp} when the assumption of an upper bound for the Ricci curvature is removed.

\begin{prop}
\label{prop1dimcounterexp}
There are constants $\epsilon ,d>0$ and a sequence of $1$-dimensional polarized pointed K\"ahler manifolds $\left( x_j , \mathbb{C}P^1 ,\mathcal{O} (1) ,\omega_j \right) $ such that $ \diam \left( \mathbb{C}P^1 ,\omega_j \right) \leq d $, $\Ric \left( \omega_j \right) \geq \epsilon \omega_j $, and the Bergman kernels satisfying that
$$ \limsup_{j\to\infty } \left( j^{-n}\rho_{\omega_j ,j } \left( x_j \right) \right) \leq 1 -\epsilon .$$
\end{prop}

\begin{proof}
For each $k\in\mathbb{N}$, let $f_k $ be defined by
\begin{displaymath}
f_k (r) = \left\{ \begin{array}{ll}
\frac{1}{2k} \sin \left( 2 k r \right) , & r\in \left[ 0, \alpha_k \right)  ,\\
\frac{\sqrt{2}}{3k} + \frac{1}{3} \sin \left( r- \alpha_k \right)  , & r\in \left[ \alpha_k , \alpha_k +\frac{\pi}{2} \right)  ,\\
\frac{k+ \sqrt{2}}{3k} \sin \left( \frac{\pi}{2} + \frac{3k\left( r-\alpha_k -\frac{\pi}{2} \right)}{k+\sqrt{2}} \right)   , & r\in \left[ \alpha_k +\frac{\pi}{2} , \alpha_k +\frac{4k +\sqrt{2}}{6k} \pi \right)  , \\
0 ,& \textrm{otherwise,}
\end{array} \right.
\end{displaymath}
where $\alpha_k =\frac{ \arccos\left( \frac{1}{3} \right) }{2k}$. It is clear that $f_k \in C^{1,1} \left( \left[ 0,  \alpha_k +\frac{4k +\sqrt{2}}{6k} \pi \right] \right) $, and there exists a constant $\kappa >0$ such that $-f''_k \geq 2\kappa f_k $ on $\left[ 0,  \alpha_k +\frac{4k +\sqrt{2}}{6k} \pi \right] $. 

By smoothing the functions $f_k$, we can conclude that there exists a sequence of nonnegative functions $\psi_{k} \in C^{2} \left( \left[ 0,  \alpha_k +\frac{4k +\sqrt{2}}{6k} \pi \right] \right) $ such that $ -\psi''_k \geq \kappa \psi_k $, and $\phi_k =f_k $ on $\left[ 0, \frac{1}{2} \alpha_k \right] \cup \left[ \alpha_k +\frac{2\pi }{3} ,  \alpha_k +\frac{4k +\sqrt{2}}{6k} \pi \right] $, $\forall k\in\mathbb{N}$. Then $dr^2 +\psi^2_k d\theta^2 $ gives a sequence of K\"ahler metrics $\omega_{1,k} $ on $\mathbb{C} P^1 \cong S^2 $, satisfying that $\Ric \left( \omega_{1,k} \right) \geq \kappa \omega_{1,k} $, $\Vol \left( \omega_{1,k} \right) \geq \frac{1}{10} $, $\forall k\in\mathbb{N}$, and the Gromov-Hausdorff limit of $\left( \mathbb{C}P^1 ,\omega_{1,k} \right)$ is a K\"ahler orbifold with only one singular point of type $\mathbb{C}/ \mathbb{Z}_3 $. By rescaling we can find a sequence of K\"ahler metrics $\omega_k $ on $\mathbb{C}P^1 $, such that $\Ric \left( \omega_k \right) \geq \frac{\kappa}{10} \omega_k $ and $\Vol \left( \omega_{k} \right) =1 $, $\forall k\in\mathbb{N}$. It follows that $\omega_{k} $ are polarized K\"ahler metrics, and we can find a sequence of domains $U_k =B_{\frac{1}{10} } \left( y_k \right) \subset \left( \mathbb{C}P^1 ,\omega_k \right) $ such that the metric $\omega_k $ can be represented by $dr^2 +\psi^2_k d\theta^2 $ on $ U_k $.

After passing to a subsequence, we can find a non-descending sequence of integers $\left\lbrace k_j \right\rbrace_{j=1}^{\infty} $ such that the sequence of functions $j^{\frac{1}{2}} \psi_{k_j} \left( j^{-\frac{1}{2}} r \right) $ converges uniformly to $\frac{r}{3}$ on $[0,R]$, $\forall R>0$. Since the metric $j\omega_{k_j} $ can be represented by $dr^2 +j\psi^2_{k_j} \left( j^{-\frac{1}{2}} r \right) d\theta^2 $ on $ U_k $, we can conclude that there are holomorphic sections $e_j \in H^0 \left( U_{k_j} ,\mathcal{O} (j) \right) $ such that $$ \lim_{j \to\infty } \sup_{U_{k_j}} \left| \left\Vert e_j \right\Vert^2 -e^{-\pi r^2} \right| =0 .$$ 
It is easy to check that the metric cone $\left( \mathbb{R}^2 , dr^2 + \frac{r^2}{9} d\theta^2 \right) $ is isomorphic to $\left( \mathbb{C} /\mathbb{Z}_3 ,\omega_{Euc} \right) $. Then this theorem follows from a similar argument as in Theorem \ref{dscounterexp}.
\end{proof}


\end{document}